\documentclass[10pt]{article}
\usepackage{authblk}
\usepackage{graphicx}
\usepackage{lipsum}
\usepackage{multirow}
\usepackage{amsfonts}
\usepackage{subfigure}
\usepackage{graphicx}
\usepackage{epstopdf}
\usepackage{color}
\usepackage{float}
\usepackage{tikz} 
\usepackage{amsmath}
\usepackage{multirow}
\usetikzlibrary{positioning,datavisualization,datavisualization.formats.functions}
\usepackage{amssymb,mathrsfs,MnSymbol}
\usepackage{amsthm}

\def\bw {\mathbf{w}}

\def\btheta {\mathbf{\theta}}
\def\bomg {\boldsymbol{\omega}}

\def\bbR {\mathbb{R}}

\def\cE {\mathcal{E}}

\def\cI {\mathcal{I}}

\def\cL {\mathcal{L}}

\def\eps {{\epsilon}}

\def\om {{\omega}}
\def\bom{{\mathbf{\omega}}}

\newcommand{\ba}{\begin{aligned}}
\newcommand{\ea}{\end{aligned}}

\newcommand{\be}{\begin{equation}}
\newcommand{\ee}{\end{equation}}

\newtheorem{Thm}{Theorem}[section]
\newtheorem{Rmk}[Thm]{Remark}

\title{A Deep Learning Based  Discontinuous Galerkin Method for Hyperbolic Equations with Discontinuous Solutions and Random Uncertainties}
\author[a]{Jingrun Chen \thanks{jingrunchen@suda.edu.cn}}
\affil[a]{School of Mathematical Sciences and Mathematical Center for Interdisciplinary Research, Soochow University, Suzhou, 215006, China}
\author[b]{Shi Jin \thanks{shijin-m@sjtu.edu.cn}}
\affil[b]{School of Mathematical Sciences, Institute of Natural Sciences, and MOE-LSC, Shanghai Jiao Tong University, Shanghai, 200240, China}
\author[c]{Liyao Lyu\thanks{lyuliyao@msu.edu}}
\affil[c]{Department of Computational Mathematics, Science, and Engineering, Michigan State University, East Lansing, MI, 48824, USA}
\begin{document}
\maketitle

\begin{abstract}

We propose a deep learning based discontinuous Galerkin method (D2GM) to solve hyperbolic equations with discontinuous solutions and random uncertainties. The main computational challenges for such problems include discontinuities of the solutions and the curse of dimensionality due to uncertainties. Deep learning techniques have been favored for
high-dimensional problems but face difficulties when the solution is not smooth, thus have so far been mainly used for
viscous hyperbolic system that admits only smooth solutions.  We alleviate this difficulty by setting up the loss function using discrete shock capturing schemes--the discontinous Galerkin method as an example--since the solutions are smooth in the discrete space.  The convergence of D2GM is established via the Lax equivalence theorem kind of argument. The high-dimensional random space
is handled by the Monte-Carlo method. Such a setup makes the D2GM approximate high-dimensional functions over the random space with satisfactory accuracy at reasonable cost. The D2GM is found numerically to be first-order and second-order accurate for (stochastic) linear conservation law with smooth solutions using piecewise constant and piecewise linear basis functions, respectively. Numerous examples are given to verify the efficiency and the robustness of D2GM with the dimensionality of random variables up to $200$ for (stochastic) linear conservation law and (stochastic) Burgers' equation.
\end{abstract}

\section{Introduction}
Hyperbolic equations with discontinuous solutions in the physical space arise in problems such as fluid mechanics, combustion, nonlinear acoustics, gas dynamics, and traffic flow \cite{Dafermos, leveque2002finite}. One famous example is the compressible Euler equations in gas dynamics, which are  
 the compressible Navier-Stokes equations without viscosity and heat conductivity.  The inviscid equations develop discontinuous solutions, aka shocks, even if one starts from smooth initial data. Capturing shock waves has been
  an important subject in scientific computing and has been very successful \cite{MR3767234, leveque2002finite}.
 Meanwhile, in reality, one may need to consider
   many sources of uncertainties that can arise in these models. They may be due to the incomplete knowledge of the model, such as the empirical equations of state or constitutive relations, imprecise measurement of physical parameters, and inaccurate measurement of boundary and initial data. Therefore, it is highly desirable to develop computational methods that not only
   capture the singular profile of solutions in the physical space but also take random uncertainties  into account in the random space for high-fidelity simulations, along the line of uncertainty quantification (UQ) \cite{jin2018uncertainty}.

Due to the high dimensionality of the problems under study, it is natural to use  deep-learning based approaches, which  have been recently proposed for high-dimensional partial differential equations; see~\cite{weinan2018deep,sirignano2018dgm,meng2020ppinn,raissi2019physics,zhang2019quantifying,weinan2020machine,2020algorithms,zang2020weak,lyu2020mim,liang2021reproducing} for examples and references therein. In these methods, the basic idea is to use a deep neural network (DNN) as the trail function to approximate the solution based on global optimization of a suitably chosen  loss function. Specifically, the parameters in the DNN are optimized to make the DNN approximation satisfy the PDE and boundary/initial conditions as accurately as possible.  Quite good approximate solutions are obtained for problems with dimensionality about $100$. In all these methods, the loss function involves the (possibly higher-order) derivatives of the PDE solution, which prevents their ability to solve problems with discontinuous solutions, such as the (inviscid) Burgers' equation and the compressible Euler equations, and hence one usually solves viscous problems in which the solutions
are smooth \cite{raissi2019physics}.

For hyperbolic equations with discontinuous solutions in the physical space, the discontinuous Galerkin (DG) method has been very popular~\cite{cockburn1989tvbo,cockburn1989tvbm,cockburn1990runge, klockner2011viscous,cockburn2012discontinuous}. The flexibility of using discontinuous basis functions makes the DG methods capable of solving equations with discontinuous solutions, such as shock waves. For such problems with uncertainties, the stochastic Galerkin (SG) method has been developed for PDEs with random coefficients \cite{barth2013uncertainty,xiu2010numerical}, such as stochastic conservation laws 
\cite{MR3493500,MR2501693, MR3674794}, stochastic Hamilton--Jacobi equation \cite{hu2015stochastic} and stochastic wave equation \cite{gottlieb2008galerkin, MR2673525}. Compared with the Monte-Carlo (MC) method, the SG method achieves the spectral accuracy given the sufficient regularity of the PDE solution in the random space. Even though the SG methods are  widely used for stochastic problems, their computational complexity grows exponentially with respect to the dimensionality of the random space. Therefore, when the dimensionality of the random space is large, the MC method is preferred.

In this work, we propose a deep learning based discontinuous Galerkin method (D2GM) to solve hyperbolic equations with discontinuous solutions and random uncertainties by combining the advantages of the DG method and DNNs. A key idea here is
that at the discrete level, the DG method as an example here, the solution is smooth although its continuous 
counterpart is not. Thus one can expect that DNN will train better than the ones using AutoGrad
in PyTorch or TensorFlow for time and/or spatial derivatives. We will give a convergence analysis for this DNN solution for the case of
1d upwind flux. The idea of taking advantage of the smoothing
effect of the  discrete derivatives  has been used previously for solving linear wave equations with
discontinuous uncertain coefficients \cite{jin2018discrete}. In the high-dimensional random space we use the MC
method. The proposed method has the following properties:
\begin{itemize}
	\item By using the DNN representation in both physical and random spaces, the D2GM can approximate the PDE solution well in high dimensions;
	\item By using the weak formulation and discontinuous element basis, the D2GM is able to approximate discontinuous PDE solutions with high accuracy;
	\item By using the mini-batch sampling with controllable number of samples, the D2GM overcomes the curse of dimensionality.
\end{itemize}
The rest of paper is organized as follows. In Section \ref{sec:D2GM}, the D2GM is proposed with details about the DNN, discontinuous element basis, loss function, boundary and initial conditions, and stochastic gradient descent method. A convergence analysis of D2GM (in 1D
and using the upwind flux) is provided in Section \ref{sec:convergence}. Numerical results with the dimensionality of random variables up to $200$ for (stochastic) linear conservation law and (stochastic) Burgers' equation are shown in Section \ref{sec:numerics}. Conclusions are drawn in Section \ref{sec:conclusion}.

\section{Deep learning based discontinuous Galerkin method}
\label{sec:D2GM}
In this section, we describe the D2GM in details. First, we introduce the construction of a DNN and build the discontinuous element space using the DNN.
The associated loss function based on the DG method is then proposed with the enforcement of boundary/initial conditions. The stochastic gradient descent method is employed to find the optimal solution.

\subsection{Deep neural network}
A DNN contains a series of layers, and each layer has several neurons linked to pre- and post- layer neurons. Neurons are connected with an affine transformation and a nonlinear activation function. Such a DNN can be viewed as a nonlinear approximation of the target function. Precisely, suppose that the DNN has $L$ layers, i.e., an input layer, $L-1$ hidden layers, and an output layer. The input layer takes $\boldsymbol{z}^0 = (t,\boldsymbol{x},\boldsymbol{\omega})$ as the input and the output layer gives $z^L=\mathcal{N}(t,\boldsymbol{x},\boldsymbol{\omega})$ as the output, where $t$ is the temporal variable, $\boldsymbol{x}$ is the spatial variable, and $\boldsymbol{\omega}$ is the random variable. The relation between the $l$-th layer and the $(l + 1)$-st layer $(l = 0, 1, ..., L-1)$ is given by
\be
\ba
\label{DNN}
& \boldsymbol{z}^0=(t, \boldsymbol{x},\boldsymbol{\omega})  \quad {\text {input}}
\\
&\mathbf{z}^{l+1}_k= \sigma_l(\bw_{k}^{l+1} \cdot \boldsymbol{z}^l+b_k^l), \quad  l=0, 1, \cdots L-1, \quad 1\le k\le m_{l+1}\,,
\\
& \mathcal{N}(t,\boldsymbol{x},\boldsymbol{\omega})= \bw^{L+1}z^L  \quad {\text {output}}
\ea
\ee
where $m_l$ is the number of neurons in the $l-$th layer ($m_L=1$),  $\sigma$ is the activation function. Some popular $\sigma$ includes the rectified linear unit (ReLU) function $\sigma(x)=\max(x,0)$ and the sigmoid function $\sigma(x)=1/(1+e^{-x})$.

Let $\btheta=(\theta_1, \cdots, \theta_J)$ include all $\bw_k^l$ and $b_k^l$, with $J$ the total number of coefficients in (\ref{DNN}), which are to be obtained by minimizing
the loss function, in order to match the DNN solution $\mathcal{N}(t,\boldsymbol{x},\boldsymbol{\omega})$ with the target function $u(t,\boldsymbol{x},\boldsymbol{\omega})$.

\subsection{Discontinuous element basis}

For brevity, we use the unit interval $[0,1]$ for demonstration. Denote
\begin{equation}
	0 = x_0 < x_{\frac{1}{2}} < x_1 <\cdots < x_{N-\frac{1}{2}}< x_N = 1,
\end{equation}
where $x_{i+\frac{1}{2}}$ is the middle point of the cell $I_i=[x_i,x_{i+1}]$.
 We also denote $\Delta x_i = x_{i+1}-x_{i}$ and $h = \max_i \Delta {x_i}$.
 For the uniform mesh, $h = \Delta x_{i} = \frac{1}{N}$. 
\begin{figure}[ht]
\centering
\begin{tikzpicture}

	\draw[->] (-0.2,0)->(10.2,0);
\foreach \x in {0,2,...,10}
{
    \draw[xshift=\x cm] (0,0) -- (0,0.2);
}; 
\foreach \x in {1,3,...,9}
{
    \draw[xshift=\x cm] (0,0) -- (0,0.1);
};  

\node[below] at(0,0){$x_0$};
\node[below] at(1,0){$x_{\frac{1}{2}}$};
\node[below] at(2,0){$x_1$};
\node[below] at(3,0){$x_\frac{3}{2}$};
\node[below] at(4,0){$x_2$};
\node[below] at(5,0){$\cdots$};
\node[below] at(6,0){$x_{N-2}$};
\node[below] at(7,0){$x_{N-\frac{3}{2}}$};
\node[below] at(8,0){$x_{N-1}$};
\node[below] at(9,0){$x_{N-\frac{1}{2}}$};
\node[below] (a) at(10,0){$x_N$};
\draw[cyan,domain=0:4,smooth] plot(\x,{sin(\x r)+0.5});
\draw[cyan,domain=6:10,smooth] plot(\x,{sin(\x r)+0.5});
\draw[cyan,domain=4:6,dashed,smooth] plot(\x,{sin(\x r)+0.5});
\draw[cyan,xshift = 10 cm, yshift = 1.5 cm] (-1,0) -- (-0.5,0) node at (0.3,0){$\mathcal{N}_\theta(x)$};
\draw[orange,xshift = 10 cm, yshift = 2 cm] (-1,0) -- (-0.5,0) node at (0.3,0){$\hat{u}_{h,\theta}$};
\foreach \x in {1,3}
{
\draw[orange,xshift = \x cm] (-1,{sin(\x r)+0.5}) -- (1,{sin(\x r)+0.5});
\draw[xshift = \x cm,dashed] (0,0) -- (0,{sin(\x r)+0.5});
\draw[xshift = \x cm,dashed] (-1,0) -- (-1,{sin(\x r)+0.5});
\draw[xshift = \x cm,dashed] (1,0) -- (1,{sin(\x r)+0.5});
}
\foreach \x in {7,9}
{
\draw[orange,xshift = \x cm] (-1,{sin(\x r)+0.5}) -- (1,{sin(\x r)+0.5});
\draw[xshift = \x cm,dashed] (0,0) -- (0,{sin(\x r)+0.5});
\draw[xshift = \x cm,dashed] (-1,0) -- (-1,{sin(\x r)+0.5});
\draw[xshift = \x cm,dashed] (1,0) -- (1,{sin(\x r)+0.5});
}

\end{tikzpicture}
\caption{Illustration of the discontinuous element space.}
\label{fig:element space}
\end{figure}
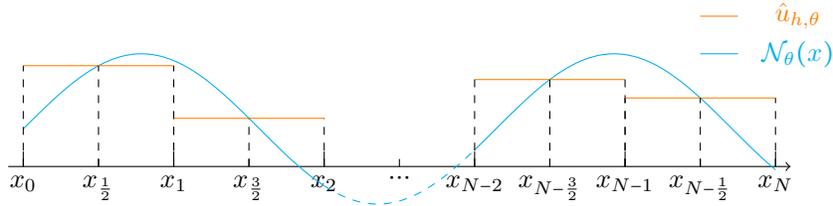

The discontinuous element space is defined as
 \begin{equation}
 	V_h^0 = \{v : v|_{I_i} \in P^0(I_i), \quad 0\leq i<N\},
 \end{equation}
 where $P^0$ denotes the $0$-th order polynomial (constant). We use a DNN to represent the element in $V_h^0$ as
 \begin{equation}
	u_{h,\theta}(x) = \mathcal{N}_{\theta}(x_{i+\frac{1}{2}}),  \quad \text{if } x_i\leq x<x_{i+1},
\end{equation}
where $\theta$ is the parameter set in the DNN to be optimized. This can also be expressed in a way more like the Galerkin formulation 
\begin{equation}\label{dnn-0}
	u_{h,\theta}(x) = \sum_{i=0}^{N-1} \mathcal{N}_{\theta}(x_{i+\frac{1}{2}}) \varphi_i(x) , \quad \varphi_i(x) = \left\{\begin{matrix}
		1 & x_i \leq x < x_{i+1}\\
		0 & \text{otherwise}
	\end{matrix}\right..
\end{equation}
This procedure is illustrated in Figure \ref{fig:element space}. This definition can be generalized to the space of high-order piecewise polynomials
\begin{equation}
\label{equ:space}
V^K_h=\left\{v:v|_{I_i}\in P^K(I_i); 1\leq i<N\right\},
\end{equation}
and any element in the space can be represented by $K+1$ DNNs $\mathcal{N}_\theta^j, j=0,\cdots,K,$ as
\begin{equation}\label{dnn-k}
u_{h,\theta}(x) = \sum_{j=0}^K \sum_{i=0}^{N-1} \mathcal{N}_\theta^j (x_{i+\frac{1}{2}})\varphi^j_i(x),
\end{equation}
where $\varphi_i^j(x)$ is the $j-$th order Legendre polynomial defined in $I_i$.

In high dimensions, this definition of $V_h^0$ can be easily generalized 
\begin{equation}
\label{equ:coeff}
\begin{aligned}
	&u_{h,\theta}(\mathbf{x}) = \sum_{\boldsymbol{i}} \mathcal{N}_{\theta}(x_{\boldsymbol{i}+\frac{1}{2}}) \varphi_{\boldsymbol{i}}(\mathbf{x}) &
	\varphi_{\boldsymbol{i}}(\mathbf{x}) = \left\{\begin{matrix}
		1 & \mathbf{x} \in I_{\boldsymbol{i}}\\
		0 & \text{otherwise}
	\end{matrix}\right.,
\end{aligned}
\end{equation}
where $\boldsymbol{x}=(x^1,x^2,\cdots, x^d)\in \mathbb{R}^d$, $\boldsymbol{i} = (i_1,i_2,\cdots i_d)$ is a multi-dimensional index vector, $I_{\boldsymbol{i}} = [x^1_{i_1},x^1_{i_1+1})\times [x^2_{i_2},x^2_{i_2+1}) \times \cdots [x^d_{i_d},x^d_{i_d+1})$, and $x_{\boldsymbol{i}+\frac{1}{2}}$ represents the center of $I_{\boldsymbol{i}}$.
The generalization of $V_h^k$ to the high-dimensional case can be done in a similar manner.

\subsection{The DG method for hyperbolic  conservation law}
In this work, we consider the hyperbolic problem with random uncertainties of the following form
\begin{eqnarray}\label{equ:conservation law}
&	u_t +\nabla_{\boldsymbol{x}} \cdot \boldsymbol{f}(u) = 0  \\
&   u(0, \boldsymbol{x}, \boldsymbol{\omega})=u_0(\boldsymbol{x}, \boldsymbol{\omega})
\end{eqnarray}
defined for $(t,\boldsymbol{x},\boldsymbol{\omega})\in [0,T]\times D\times\Omega$. Here
$\bomg$ is a high-dimensional random variable representing uncertainties (or random inputs).
The solution $u=u(t,\boldsymbol{x},\boldsymbol{\omega})$ then depends on $\bomg$.

The semi-discrete DG method for solving \eqref{equ:conservation law} is defined as follows: Find the unique solution $u_h(t,\boldsymbol{x},\boldsymbol{\omega})\in V_h^k$, such that, for any test function $v_h\in V^k_h$ and all $0\leq i <N$, one has
\begin{multline}
\label{equ:semi-dis}
	\frac{\mathrm{d}}{\mathrm{d} t}\left( u_h(t,\boldsymbol{x},\boldsymbol{\omega}),v_h(\boldsymbol{x})\right)_{I_{\boldsymbol{i}}} - \left(\boldsymbol{f}(u_h(t,\boldsymbol{x},\boldsymbol{\omega})),\nabla v_h(\boldsymbol{x})\right)_{I_{\boldsymbol{i}}} \\
	+ \boldsymbol{f}(u_h(t,\boldsymbol{x},\boldsymbol{\omega}))\cdot\mathbf{n} v_h(\boldsymbol{x}))|_{\partial I_{\boldsymbol{i}}} = 0,
\end{multline}
where $\boldsymbol{n}$ is the outward unit normal vector along $\partial I_{\boldsymbol{i}}$.
In 1D, one has
\begin{equation*}
	\frac{\mathrm{d}}{\mathrm{d} t}\left( u_h(t,x,\boldsymbol{\omega}),v_h(x)\right)_{I_i} - \left(f(u_h(t,x,\boldsymbol{\omega})),v'_h(x)\right)_{I_i} +  \hat{f}_{i+1}v_h(x_{i+1}^-) - \hat{f}_{i}v_h(x_{i}^+) = 0,
\end{equation*}
where the one-sided limit is defined as 
\begin{equation*}
	v^\pm(x_j) = v(x_j^{\pm}) = \lim_{x\to x_j^\pm} v(x),
\end{equation*}
and the inner product is defined as 
\[
(a(t,x,\boldsymbol{\omega}),b(x))_{D} = \int_{D} a(t,x,\boldsymbol{\omega})\cdot b(x) \mathrm{d} x.
\] 
Here $\hat{f}_i$ is a numerical flux, which is a single-valued function defined at the interface $x_i$ and in general depends on the values of the numerical solution $u_h$ from both sides of the interface.
There are several choices to choose the flux and we use the upwind flux 
\begin{equation*}
\hat{f}^{\mathrm{upwind}}(u^-,u^+) = f(u^-)
\end{equation*}
and Godunov flux 
\begin{equation*}
	\hat{f}^{\mathrm{God}}(u^-,u^+) = \left\{
	\begin{matrix}
	&\min_{u^-\leq u\leq u^+} f(u),& \text{if } u^-<u^+\\
	&\max_{u^+\leq u\leq u^-} f(u),& \text{if } u^+<u^+\\
	\end{matrix}\right.	
\end{equation*}
in this work~\cite{chavent1989local}. In high dimensions, $\boldsymbol{f}(u_h(t,\boldsymbol{x},\boldsymbol{\omega}))$ in $(\boldsymbol{f}(u_h(t,\boldsymbol{x},\boldsymbol{\omega})),\mathbf{n} v_h(\boldsymbol{x}))|_{\partial I_{\boldsymbol{i}}}$ is replaced by the numerical flux on quadrature points.

The semi-discrete formulation \eqref{equ:semi-dis} includes the temporal derivative, which needs to be discretized. A simple idea is to use the AutoGrad in PyTorch or TensorFlow, which provides the temporal derivative automatically by back propagation. This kind of approach is widely used in solving PDEs with spatial derivatives evolved in the loss function \cite{weinan2018deep,sirignano2018dgm}. We also introduce the temporal discretization with steps $0=t_0<t_1<\cdots<t_N= T$ and $t_{n+1}-t_n= \Delta t$, and the semi-discrete formulation \eqref{equ:semi-dis} becomes
\begin{multline}
	\label{equ:weak form high dimension}
\left( \frac{u_h(t_{n+1},\boldsymbol{x},\boldsymbol{\omega})-u_h(t_n,\boldsymbol{x},\boldsymbol{\omega})}{ \Delta t},v_h(\boldsymbol{x})\right)_{I_{\boldsymbol{i}}} - \left(\boldsymbol{f}(u_h(t_n,\boldsymbol{x},\boldsymbol{\omega})),\nabla v_h(\boldsymbol{x})\right)_{I_{\boldsymbol{i}}}\\
	 + (\boldsymbol{f}(u_h(t,\boldsymbol{x},\boldsymbol{\omega})),\mathbf{n} v_h(\boldsymbol{x}))|_{\partial I_{\boldsymbol{i}}}  = 0. 
\end{multline}
In 1D, \eqref{equ:weak form high dimension} reduces to
\begin{multline}
\label{equ:weak form}
\left( \frac{u_h(t_{n+1},x,\boldsymbol{\omega})-u_h(t_n,x,\boldsymbol{\omega})}{\Delta t},v_h(x)\right)_{I_i} - \left(f(u_h(t_n,x,\boldsymbol{\omega})), v'_h(x)\right)_{I_i}\\
 +  \hat{f}_{i+1}v_h(x_{i+1}^-) - \hat{f}_{i}v_h(x_{i}^+) = 0. 
\end{multline}

Consider the DG approximation
$$u_{h,\theta}(t,x,\boldsymbol{\omega}) = \sum_{j'=0}^K \sum_{i'=1}^{N-1}  \mathcal{N}_\theta^{j'} (t,x_{i'+\frac{1}{2}},\boldsymbol{\omega})\varphi^{j'}_{i'}(x)\,.
$$
Substituting  this into \eqref{equ:weak form}, choosing $v_h(x)=\varphi^{j}_{i}(x)$, and using the orthogonality of the Legendre polynomials, we have
\begin{multline}\label{eqn:discrete}
L_{i,j,n}\triangleq\frac{\mathcal{N}^j_\theta(t_{n+1},x_{i+\frac{1}{2}},\boldsymbol{\omega})-\mathcal{N}^j_\theta(t_{n},x_{i+\frac{1}{2}},\boldsymbol{\omega})}{\Delta t} - \left(f(u_{h,\theta}(t_n,x,\boldsymbol{\omega})),\frac{\mathrm{d} \varphi^j_i(x)}{\mathrm{d} x}\right)_{I_i} \\
+  \hat{f}_{i+1}\varphi^j_i(x_{i+1}^-) - \hat{f}_{i} \varphi^j_i(x_{i}^+) = 0 	.
\end{multline}
The second term above can be further simplified when $f$ is specified. For example, for linear conservation law when $f(u) = u$,   
\begin{equation}
\label{equ:linear loss V2}
\left(f(u_{h,\theta}(t_n,x,\boldsymbol{\omega})),\frac{\mathrm{d} \varphi^j_i(x)}{\mathrm{d} x}\right)_{I_i} = \sum_{k=0}^K\mathcal{N}^k_\theta(t_{n},x_{i+\frac{1}{2}},\boldsymbol{\omega}) C_k^j,\quad j=0,\cdots, K,
\end{equation}
where $C_k^j = \left(\varphi^k_i(x),\frac{\mathrm{d} \varphi^j_i(x)}{\mathrm{d} x}\right)_{I_i}$, and for Burgers' equation when $f(u) = \frac{1}{2}u^2$,
$$
\left(f(u_{h,\theta}(t_n,x,\boldsymbol{\omega})),\frac{\mathrm{d} \varphi^j_i(x)}{\mathrm{d} x}\right)_{I_i} = \sum_{l=0}^K\sum_{l'=0}^K\mathcal{N}^l_\theta(t_{n},x_{i+\frac{1}{2}},\boldsymbol{\omega})\mathcal{N}^{l'}_\theta(t_{n},x_{i+\frac{1}{2}},\boldsymbol{\omega}) C_{l,l'}^j,
$$
where $C_{l,l'}^j = \left(\varphi^l_i(x)\varphi^{l'}_i(x),\frac{\mathrm{d} \varphi^j_i(x)}{\mathrm{d} x}\right)$. 
Therefore, the loss function  for the DNN is defined as the residual error of \eqref{eqn:discrete} in the $L^2$ sense
\begin{equation}
\label{equ:loss function}
	\mathcal{L}(\theta) = \left(h \,\Delta t \sum_{i,j,n} L^2_{i,n} \right)^{1/2},
\end{equation}
and the DNN solution $\mathcal{N}_\theta$ is the solution that minimizes the loss function:
\begin{equation}\label{mini}
  \min_\theta 	\mathcal{L}(\theta)\,.
\end{equation}
Note that in this model the random variable $\boldsymbol{\omega}$ is still continuous and no discretization is applied in the random space. In numerical experiments, we apply the MC method for the random variable.

\subsection{Boundary and initial conditions}
Boundary and initial conditions are important for the well-posedness of PDEs. In general, there are four kinds of boundary conditions:
\begin{itemize}
	\item Dirichlet boundary condition
	\begin{equation*}\label{equ:BC}
		u(\boldsymbol{x}, \bomg) = g(\boldsymbol{x}, \bomg) \quad \boldsymbol{x} \in \partial D.
	\end{equation*}
	\item Neumann boundary condition
		\begin{equation*}
		\frac{\partial u(\boldsymbol{x}, \bomg)}{\partial \boldsymbol{\nu} } = g(\boldsymbol{x}, \bomg) \quad \boldsymbol{x} \in \partial D,
	\end{equation*}
	where $\boldsymbol{\nu}$ is the outward unit normal vector along $\partial D$.
	\item Robin boundary condition
	\begin{equation*}
		u(\boldsymbol{x}, \bomg) + \frac{\partial u(\boldsymbol{x}, \bomg)}{\partial \boldsymbol{\nu}} = g(\boldsymbol{x}, \bomg) \quad \boldsymbol{x} \in \partial D.
	\end{equation*}
	\item Periodic boundary condition
	\begin{equation*}
		u(\boldsymbol{x}+L_i\boldsymbol{e}_i, \bomg) = u(\boldsymbol{x}, \bomg) \quad i=1,\cdots,d,\;\boldsymbol{x} \in D,
	\end{equation*}
	where $\boldsymbol{e}_i$ is the $i-$th standard unit vector and $L_i$ is the period along $\boldsymbol{e}_i$.
\end{itemize} 
There are a couple of ways to enforce boundary conditions. The most straightforward way is to add the penalty term into the loss function.  For example, the penalty term for Dirichlet boundary condition can be expressed as $\lambda\|u-g\|^2_{\partial D}$  with the penalty parameter $\lambda$. Another way is to build a DNN that satisfies the boundary condition exactly. For Dirichlet boundary condition, such a DNN can be constructed as
\begin{equation}
	u_{\theta}(t,\boldsymbol{x}, \bomg) =L(\boldsymbol{x}) \mathcal{N}_{\theta}(t,\boldsymbol{x}, \bomg) + G(\boldsymbol{x}, \bomg),
\end{equation}
where $L(\boldsymbol{x})$ is a distance function that takes $0$ on $\partial D$ and is strictly positive inside $D$, $\mathcal{N}_{\theta}(t,\boldsymbol{x})$ is the neural network, and $G(\boldsymbol{x})$ is a smooth extension of $g(\boldsymbol{x})$ and equals $g(\boldsymbol{x})$ on $\partial D$.

For the initial condition $u(0,\boldsymbol{x}, \bomg) = u_0(\boldsymbol{x}, \bomg)$, one can use
\begin{equation}\label{eqn:dnn-initial}
\begin{aligned}
	u_{\theta}(t,\boldsymbol{x}, \bomg) = t\mathcal{N}_{\theta}(t,\boldsymbol{x}, \bomg) + h(\boldsymbol{x}, \bomg).
	\end{aligned}
\end{equation}
 
In the current work, we can enforce the exact boundary condition on the numerical solution. For Neumann and periodic boundary conditions, we have
\begin{itemize}
\item Neumann (Reflecting) boundary condition	
	\begin{equation*}
		\begin{aligned}
			f_{0} & = f_{1},\\
			f_{N} & = f_{N-1}.\\
		\end{aligned}
	\end{equation*}
\item Periodic boundary condition
	\begin{equation}\label{equ:periodic}
		\begin{aligned}
			f_{0} & = f_{N-1},\\
			f_{N} & = f_{1}.\\
		\end{aligned}
	\end{equation}
\end{itemize}
This kind of approach only applies for a grid-based method. 

\subsection{Stochastic gradient descent method}
Stability and convergence of the DG method defined in Section \ref{sec:D2GM} has been studied thoroughly in classical numerical analysis \cite{cockburn2012discontinuous}. However, when it comes to the high dimension, the number of degrees of freedom (dofs) scales like $(1/h)^d$ with $d$ the dimensionality. Therefore, the classical method suffers from the curse of dimensionality. To overcome this difficulty, we apply the idea of stochastic gradient descend (SGD) method to evaluate the loss function \eqref{equ:loss function} by selecting mesh points randomly over the index set $i, j, k$ in each iteration with a fixed number of points. For the random variable, we also apply the MC method with a fixed number of points in the random space. Overall, the proposed method overcomes the curse of dimenisonality by design.
 
\section{Convergence}\label{sec:convergence}
The convergence of the DNN solution can be established through standard Lax equivalence theorem kind of augument: consistency and stability imply convergence.
We first state some preparation results which give consistency of the DNN approximation. The main reason that the DNN approximation (\ref{DNN}) works is because of the {\it universal approximation
theorem}, established in \cite{Cyb, Fun}.

To make the presentation simple and clear,  we consider the deterministic (no $\bomg$ dependence) equation \eqref{equ:conservation law} over $[0,T]\times[0,1]$ with periodic boundary condition and assume that $f\in C^1$ and $f'>0$, thus the upwind scheme on uniform mesh writes

\be
\ba
\label{upwind}
&\frac{U^{n+1}_i-U^n_i}{\Delta t} + \frac{f(U^n_i)-f(U^n_{i-1}) }{h} = 0\,, \qquad i=1, \cdots, I, \quad n=0, \cdots, N-1,
\\
& U^n_0=U^n_I,
\\
&U^0_i=u_0(x_i).
\ea
\ee

In this section we will provide a proof of the convergence of the deep neural network approximation, along the line of \cite{HJJL}.  
Consider $V(t,x)$, the solution to 
\be
\ba
\label{V-scheme}
& \frac{V(t+\Delta t, x)-V(t, x)}{\Delta t} + \frac{f(V(t, x))-f(V(t, x-h) }{h} = 0, \qquad  n\ge 0,\\
& V(t, 0)=V(t, 1),\\
& V(0, x)=u_0(x).
\ea
\ee
Without loss of generality assume $u_0(x)\in C^1(D)$, with $D=[0,1]$ (if not the case one can interpolate through $U_i^0$ to get
a $C^1$ function $V(0,x)$). For fixed $\Delta t$ and $h$, clearly (\ref{V-scheme}) implies that $V(t_n, x)\in C^1(D)$ for all $n\ge 0$, since $f\in C^1$.  

From the definition of $V$ clearly $V(t^n, x_i)=U_i^n$ for all $n\ge 0, 1\le i\le I$.

The loss function (\ref{equ:loss function}) is now
\begin{equation}\label{loss-1}
\mathcal{L}(\theta)=\left(h \, \Delta t \sum_{i=1}^I\sum_{n=0}^{N-1} \left|\frac{\mathcal{N}_{\theta}(t_n+\Delta t, x_i)-\mathcal{N}_{\theta}(t_n, x_i)}{\Delta t} + \frac{f(\mathcal{N}_{\theta}(t_n, x_i))-f(\mathcal{N}_{\theta}(t_n, x_i-h) }{h}\right|^2 \right)^{1/2}\,.
\end{equation}

Below we adopt the universal approximation theory to our setting.

\begin{Thm} \label{LiX} 
Let $\sigma$ be any non-polynomial function in $C^1(\bbR)$. Then for any $\delta>0$, there is a network (\ref{DNN}) such that
$$
\|V -  \mathcal{N}_{\theta}\|_{W^{1,\infty}(K)} < \delta\,.
$$
\end{Thm}

The next theorem establishes the consistency of the DNN approximation.

\begin{Thm}
\label{loss-conv}
Assume that the number of layers $L=2$ and that the solution $V$ to (\ref{V-scheme}) belongs to ${C}^1([0,T] \times [0,1])$, and the activation function
$\sigma(x) \in C^{2}$ is non-polynomial. Then for any $\delta>0$,  there exists $\theta$ and a sequence of the DNN solutions, denoted by $\mathcal{N}_{\theta}=\mathcal{N}(t,x; \theta)$,
such that when the number of parameters is sufficiently large, 
$$
|\cL [\mathcal N_{\theta}] | <C(T)  \delta
$$
for some positive constant $C(T)$ that may depend on $T$.
\end{Thm}

\begin{proof}
By \eqref{V-scheme},
\be
\ba
\cI_n(x,\theta)= & \frac{\mathcal N_{\theta}(t+\Delta t, x; \theta)-\mathcal N_{\theta}(t, x; \theta)}{\Delta t} + \frac{f(\mathcal N_{\theta}(t, x; \theta))-f(\mathcal N_{\theta}(t, x- h; \theta) }{h}
\\
& =\frac{[\mathcal N_{\theta}(t+\Delta t, x; \theta)-V(t+\Delta t, x)]-[\mathcal N_{\theta}(t, x; \theta)-V(t, x)}{\Delta t} \\
& \quad + \frac{[f(\mathcal N_{\theta}(t, x; \theta))-f(V(t, x))]-[f(\mathcal N_{\theta}(t, x-h; \theta)) - f(V(t, x-h))]}{h} \,.
\ea
\ee
Given any $\delta$, by Theorem \ref{LiX}, for $J$ sufficiently large, $\cI_n(x,\theta)$  can obviously be bounded by $\delta$ multiplied by a constant $C$ uniformly in $x$ and $n$, namely
\be
\| \cI_n(\cdot, \theta) \|_{l^\infty} \le C \delta\,.
\ee
Thus, by the Cauchy-Scharwtz inequality and the boundedness of $D$, the loss function in \eqref{loss-1} can be bounded by $\delta$ multiplied by a constant that depends on  $|D|$ and $T$ as
\begin{equation}
\mathcal{L}(\theta)=\left(h \, \Delta t \sum_{i,n} \left|\mathcal{I}_n\right|^2 \right)^{1/2}
\le C h\,IN  \Delta t\,\delta\le CT \delta
\end{equation}
since $Ih=1$ and $n \, \Delta t \le T$.
\end{proof}

The above theorem shows that one can find the parameter $\theta$ such that the loss function convergences to zero. This shows the {\it consistency} of the DNN approximation. In fact the loss function $\cL$ 
can be viewed as the truncation error of the DNN approximation, which will be made clear in the proof of Theorem \ref{Conv}.  Note that Theorem \ref{loss-conv}  does not imply that $\mathcal N $ converges to the solution of
the original problem (\ref{equ:conservation law}). Next we prove the convergence of the DNN approximation, based on the {\it stability} argument.

\begin{Thm}
\label{Conv}
Let $\theta_J$ be the sequence defined in Theorem \ref{loss-conv}, and let $\mathcal N_{\theta}$ be the solution to \eqref{mini} and $V$ be the classical numerical solution to \eqref{V-scheme}, then
$$
\| \mathcal{N}_\theta(t_n,\cdot; \theta_J)-V(\cdot)\| \le \| \mathcal{N}_\theta(0,\cdot; \theta)-u(0, \cdot)\| + C(T)\delta \,. 
$$
Consequently,
$$
| \mathcal{N}_\theta(t_n, x_i; \theta_J)-U_j^n\| \le \| \mathcal{N}_\theta(0,\cdot; \theta)-u(0, \cdot)\| + C(T)\delta \,,
\quad {\hbox {for}}\quad 1\le i\le I, \quad n>0\,.
$$
\end{Thm}

\begin{proof}
Let $\cE_i^n(\theta)=\mathcal{N}_\theta(t_n,x_i; \theta_J)- U_j^n=\mathcal{N}_\theta(t_n,x_i; \theta_J)- V(t^n,x_i)=\mathcal{N}_i^n-Vi^n$.  Clearly, one has
$$
\frac{\cE^{n+1}_i-\cE^{n}_i}{\Delta t} + \frac{f(\mathcal{N}_i^n)-f(\mathcal{N}^n_{i-1}) }{h}- \frac{f(u^n_i)-f(u^n_{i-1}) }{h}=\cI_n(x_i, \theta).
$$
Let $\lambda = \Delta t/h$. Thus
\be
\ba
\cE^{n+1}_i&=\cE^{n}_i+ \lambda \left[ (f(u^n_i)-f(\mathcal{N}_i^n))-(f(u^n_{i-1})-f(\mathcal{N}^n_{i-1})) \right]+\Delta t \cI_n(x_i, \theta) \\
&=\cE^{n}_i+ \lambda \left[ -f'(\eta^n_i)\cE_j^n+f'(\eta^n_{i-1})\cE_{i-1}^n\right] +\Delta t\cI_n(x_i, \theta) \\
&= (1-\lambda f'(\eta^n_i))\cE_i^n + \lambda f'(\eta^n_{i-1})\cE_{i-1}^n +\Delta t\cI_n(x_i, \theta), 
\ea
\ee
where $\eta_j^n$ is a point between $\mathcal{N}_i^n$ and $u_j^n$. Now taking the $l^1$ norm on the above equality,  assuming 
$$
  \lambda\, sup_\eta f'(\eta) \le 1,
$$
and using the periodic boundary condition, one gets, for all $n\ge 0$,
\be
\ba
\|\cE^{n+1}\|_{l^1}= 
 &\frac{1}{I} \sum_{i=1}^I|\cE_i^{n+1}| \\
&\le  \frac{1}{I}\sum_{i=1}^I (1-\lambda f'(\eta_i^n)) |\cE_i^n|
                + \frac{1}{I}\sum_{i=1}^I \lambda f'(\eta_{i-1}^n) |\cE_{i-1}^n|+ \Delta t \|\cI_n(\cdot, \theta)\|_{l^1}\\
&\le  \frac{1}{I}\sum_{i=1}^I (1-\lambda f'(\eta_i^n)) |\cE_i^n|
                + \frac{1}{I}\sum_{i=1}^I \lambda f'(\eta_{i}^n) |\cE_{i}^n|+ \Delta t \|\cI_n(\cdot, \theta)\|_{l^1}\\
&= \|\cE^n\|_{l^1} + C \delta \Delta t.
\ea
\ee
Consequently one has
\be
\|\cE^{n}\|_{l^1} \le \|\cE^{0}\|_{l^1}+ C \delta n \Delta t 
\le \|\cE^{0}\|_{l^1}+  CT \delta  
\ee
for all $n$ such that $n \Delta t \le T$. Now the convergence $\mathcal{N}\to V$ as $J\to \infty$ is a consequence of Theorem \ref{loss-conv}, as long as one trains the initial data $\cE^{0}$ well.
\end{proof}

\begin{Rmk}
It is straightforward to establish the convergence between the numerical solution $V$ and the exact solution $u(t,x)$ by combining classical numerical analysis of the DG method \cite{cockburn2012discontinuous}. This, together with Theorem \ref{Conv}, leads to the convergence of the DNN solution $\mathcal{N}_{\theta}$ to the exact solution $u(t,x)$. The proof also implies that both the DNN approximation error and the discretization error contribute to the approximation error in the D2GM, as demonstrated in Section \ref{sec:numerics}.
\end{Rmk}

\begin{Rmk}
The convergence analysis can be generalized to high-dimensional problems with random variables. In these cases, the MC method is employed to sample the random variables or a subset of indices for the spatial variables. The sampling error is inversely proportional to the square root of the number of samples, which attributes to the convergence of the loss function proven in Theorem \ref{loss-conv} while the other parts of the convergence proof remains unchanged.
In practice, the sampling error contributes to the total approximation error and is kept small by using a large number of samples.
\end{Rmk}

\section{Numerical results}
\label{sec:numerics}
There are four sources of error in the D2GM: the DNN approximation error, the discretization error, the optimization error, and the sampling error. A large number of samples are used so that the sampling error will not affect the observation numerically. For the optimization error, Adam (Adaptive moment method) is used to find the optimal solution. Therefore, the first two sources of error dominates the numerical performance of the D2GM. For a DNN with the large number of parameters, the DNN approximation error is small and the discretization error dominates. Therefore, for moderate grid size, the convergence rate of D2GM is observed in the classical sense. When the grid size is small, the DNN approximation contributes more to the total error and the convergence rate of DG will be lost. The convergence rate will be recovered if a DNN with more parameters is employed.

 \subsection{Linear conservation law}
 Consider 
 \begin{equation}\label{eqn:linear}
 \left\{
 \begin{aligned}
 	 &2d\pi u_t - \sum_{i=1}^d u_{x_i} = 0 & x\in [0,1]^d\\
 	 &u(0,x) = h(x) = \sin( 2 \pi \sum_{k=1}^d x^k)
 \end{aligned}\right.
 \end{equation}
 with periodic boundary condition, and the exact solution  $u(t,x) = \sin(t + 2 \pi \sum_{k=1}^d x^k)$, $d = 1,2,3$. For the first-order method, following \eqref{eqn:dnn-initial}, we construct the  numerical solution that satisfies the initial condition exactly
\begin{equation}
\label{equ:construction}
\begin{aligned}
	&u_{h,\theta}(t,\boldsymbol{x}) = \sum_{\boldsymbol{i}} [t\mathcal{N}_\theta(t,x_{\boldsymbol{i}+\frac{1}{2}}) + g(x_{\boldsymbol{i}+\frac{1}{2}} )]\varphi_{\boldsymbol{i}}(\boldsymbol{x}) &
	\varphi_{\boldsymbol{i}}(\boldsymbol{x}) = \left\{\begin{matrix}
		1 & \boldsymbol{x} \in I_{\boldsymbol{i}}\\
		0 & \text{otherwise}
	\end{matrix}\right.
\end{aligned}
\end{equation}
 and enforce the periodic boundary condition according to \eqref{equ:periodic}. Following \eqref{equ:coeff}, we define the DNN represented coefficients as $U_i(t) = t \mathcal{N}(t,x_{i+\frac{1}{2}}) + \sin(2\pi x_{i+\frac{1}{2}})$ in 1D, $U_{i_1,i_2}(t) = t \mathcal{N}(t,x^1_{i_1+\frac{1}{2}},x^2_{i_2+\frac{1}{2}}) + \sin(2\pi (x^1_{i_1+\frac{1}{2}}+x^2_{i_2+\frac{1}{2}} ) )$ in 2D, and  $U_{i_1,i_2,i_3}(t) = t \mathcal{N}(t,x^1_{i_1+\frac{1}{2}},,x^2_{i_2+\frac{1}{2}},x^3_{i_3+\frac{1}{2}}) + \sin(2\pi (x^1_{i_1+\frac{1}{2}}+x^2_{i_2+\frac{1}{2}} +x^3_{i_3+\frac{1}{2}} ) )$ in 3D, respectively.
 
For the upwind scheme, in 3D, the loss function for the semi-discrete scheme and the fully discrete scheme based on the forward Euler method as
\begin{equation}\label{loss-semi}
 \begin{aligned}
 	 	\mathcal{L}_{\mathrm{semi}}(\theta) &=  \Bigg(\Delta t h^3\sum_{i_i,i_2,i_3,j}\bigg(6\pi \partial_t U_{i_1,i_2,i_3}(t_j)  - \frac{U_{i_1+1,i_2,i_3}(t_j) -U_{i_1,i_2,i_3}(t_j)}{h} \\ 
		& - \frac{U_{i_1,i_2+1,i_3}(t_j) -U_{i_1,i_2,i_3}(t_j)}{h} - \frac{U_{i_1,i_2,i_3+1}(t_j) - U_{i_1,i_2,i_3}(t_j)}{h}  \bigg)^2\Bigg)^{1/2},
 \end{aligned}
 \end{equation}
and
\begin{equation}\label{loss-fe}
\begin{aligned}
\mathcal{L}_{\mathrm{FE}}(\theta) &= \Bigg(\Delta t h^3 \sum_{i_i,i_2,i_3,j}\bigg(6\pi \frac{ U_{i_1,i_2,i_3}(t_{j+1}) - U_{i_1,i_2,i_3}(t_j)}{\Delta t} - \frac{U_{i_1+1,i_2,i_3}(t_j) -U_{i_1,i_2,i_3}(t_j)}{h} \\  
& - \frac{U_{i_1,i_2+1,i_3}(t_j)  -U_{i_1,i_2,i_3}(t_j)}{h}  - \frac{U_{i_1,i_2,i_3+1}(t_j)-U_{i_1,i_2,i_3}(t_j)}{h}  \bigg)^2\Bigg)^{1/2},
\end{aligned}
\end{equation}
respectively.

Numerical results of both loss functions are recorded in Table \ref{tbl:linear cv law}. The fully discrete method based on the forward Euler scheme \eqref{loss-fe} shows a better approximation accuracy than the semi-discrete method using AutoGrad \eqref{loss-semi}. This implies that use of discrete derivative may lead to better results for time-dependent problems. For moderate mesh size $h$, the first-order convergence is observed with respect to $h$. For smaller $h$, the first-order convergence is lost but is recovered when a wider network with the width $200$ is employed; see Table \ref{tbl:linear cv law wide} for details.
 \begin{table}[H]
 \centering
 \begin{tabular}{|c|c|c|c|c|c|}
 \hline
  \multirow{2}{*}{d}	& \multirow{2}{*}{$h = \Delta t $}&  \multicolumn{2}{c|}{Fully discrete} & \multicolumn{2}{c|}{Semi-discrete} \\
 \cline{3-6}
  & & error & order & error & order\\
 \hline
 \multirow{6}{*}{1} & 1/10 & 2.86 e-01 &       & 3.04 e-01 & \\
   & 1/20 & 1.50 e-01 & 0.93 & 1.58 e-01 & 0.93\\
   & 1/40 & 7.73 e-02 & 0.94 & 8.05 e-02 & 0.98\\
   & 1/80 & 3.95 e-02 & 0.96 & 8.39 e-02 & -0.05\\
   & 1/160& 2.10 e-02 & 0.91 & 7.01 e-02 & 0.25\\
   & 1/320& 1.72 e-02 & 0.28 & 1.44 e-01 & -1.04\\
 \hline
 \multirow{6}{*}{2} & 1/10 & 3.32 e-01 &       & 3.43 e-01 & \\
   & 1/20 & 1.72 e-01 & 0.90  & 1.81 e-01 & 0.91\\
   & 1/40 & 8.90 e-02 & 0.95  & 8.89 e-02 & 1.03\\
   & 1/80 & 4.68 e-02 & 0.92  & 6.00 e-02 & 0.56\\
   & 1/160& 2.57 e-02 & 0.86  & 5.40 e-02 & 0.15\\
   & 1/320& 1.87 e-02 & 0.45  & 5.64 e-02 & -0.06\\
   \hline
 \multirow{6}{*}{3} & 1/10 & 3.59 e-01 &       & 3.75 e-01 & \\
   & 1/20 & 1.92 e-01 & 0.90  & 2.02 e-01 & 0.89\\
   & 1/40 & 9.95 e-02 & 0.95  & 1.03 e-01 & 0.96\\
   & 1/80 & 5.20 e-02 & 0.93  & 6.94 e-02 & 0.57\\
   & 1/160& 3.14 e-02 & 0.72  & 9.45 e-02 & -0.44\\
   & 1/320& 2.09 e-02 & 0.58  & 1.16 e-01 & -0.30\\
 	\hline
 \end{tabular}
 \caption{The averaged $L^2$ relative error in the last 1000 steps and the convergence rate for the linear conservation law \eqref{eqn:linear} with two loss functions \eqref{loss-semi} and \eqref{loss-fe}. A neural network with $4$ hidden layers and two shortcut connections is used and the batchsize is chosen as $10000$. In 1D, the network width is set to be $20$ and the total number of parameters is $1341$. In 2D, the network width is set to $40$ and the total number of parameters is $5121$. In 3D, the network width is set to be $60$ and the total number of parameters is $11341$.}
 \label{tbl:linear cv law}
 \end{table} 
\begin{table}[H]
 	\centering 
 	\begin{tabular}{|c|c|c|}
 	\hline
 	$h = \Delta t $ & error & order \\
 	\hline
 	1/10 & 3.64 e-01 & \\
	1/20 & 1.92 e-01 & 0.92\\
	1/40 & 9.92 e-02 & 0.95\\
	1/80 & 5.04 e-02 & 0.97\\
	1/160 & 2.54 e-02 & 0.98\\
	1/320 & 1.29 e-02 & 0.97\\
	1/640 & 6.78 e-03 & 0.93\\
	1/1280 &4.24 e-03 & 0.67\\
	\hline
 	\end{tabular}
 	\caption{The averaged $L^2$ relative error in the last $1000$ steps and the convergence rate for the linear conservation law \eqref{eqn:linear} using loss function \eqref{loss-fe} and a wider neural network with width $200$ in 3D.}\label{tbl:linear cv law wide}
 \end{table}
 
 For the second-order method, the approximate solution in 1D is constructed as
 \begin{equation}
 \begin{aligned}
	&u_{h,\theta}(t,x) = \sum_{i} [U_i^0\varphi^0_{i}(x) + U_i^1\varphi^1_{i}(x) ],
\end{aligned}
\end{equation}
where
 \begin{equation}
 \begin{aligned}
 	&
	\varphi^0_{i}(x) = \left\{\begin{matrix}
		1 & x \in I_{i}\\
		0 & \text{otherwise}
	\end{matrix}\right. , \\
	&\varphi^1_{i}(x) = \left\{\begin{matrix}
		(x-x_{i+\frac{1}{2}}) & x \in I_{i}\\
		0 & \text{otherwise}
	\end{matrix}\right. ,
	\\
 	&U^0_i(t) = t\mathcal{N}^0_{\theta_0}(t,x_{i+\frac{1}{2}}) + \sin(2\pi x_{i+\frac{1}{2}}),\\
 	&U^1_i(t) = t\mathcal{N}^1_{\theta_1}(t,x_{i+\frac{1}{2}}) + 2\pi\cos(2\pi x_{i+\frac{1}{2}}).
 	\end{aligned}
 \end{equation}

Based on \eqref{equ:linear loss V2}, the corresponding loss function consists of two contributions
 \begin{equation}
 \begin{aligned}
	&\mathcal{L}_0(\theta_0) = \Bigg(\Delta t h\sum_{i,j}\left(2\pi\frac{U_i^0(t_{j+1} ) - U_i^0(t_j)}{\Delta t}  h + \hat{f}_{i+\frac{3}{2}} - \hat{f}_{i+\frac{1}{2}} \right)^2\Bigg)^{1/2},\\
	&\mathcal{L}_1(\theta_1) = \Bigg(\Delta t h\sum_{i,j}\left(2\pi\frac{U_i^1(t+\Delta t ) - U_i^1(t)}{\Delta t}  \frac{h^3}{12} + hu_0  + \frac{h}{2}\hat{f}_{i+\frac{3}{2}} +\frac{h}{2} \hat{f}_{i+\frac{1}{2}} \right)^2\Bigg)^{1/2}.
 \end{aligned}
 \end{equation}
Often $\mathcal{L}_0(\theta_0)$ and $\mathcal{L}_1(\theta_1)$ are not of the same order of magnitude, which adds additional difficulties to minimize both terms simultaneously
\begin{equation}
	\arg\min_{\{\theta_0,\theta_1\}} \mathcal{L}_0(\theta_0) + \mathcal{L}_1(\theta_1),
\end{equation}
where $\theta_0,\theta_1$ are the parameters of neural networks to approximate $U^0$ and $U^1$, respectively. We use ADMM~\cite{boyd2011distributed} to optimize $U^0$ and $U^1$. Numerical results are shown in Table \ref{tbl:linear cv law high}. Compared with Table \ref{tbl:linear cv law}, we can find the second-order scheme has a better accuracy when two identical networks are applied. If a wider and deeper network is employed, then the second-order scheme is obtained with high accuracy; see Table \ref{tbl:linear cv law high_deep}.
\begin{table}[H]
 	\centering 
 	\begin{tabular}{|c|c|c|}
 	\hline
 	$h = \sqrt{\Delta t}$ & error & order \\
 	\hline
 	1/10 & 1.01 e-01 & \\
	1/20 & 3.69 e-02 & 1.40\\
	1/40 & 2.24 e-02 & 0.79\\
	1/80 & 1.15 e-02 & 0.95\\
	1/160 & 1.06 e-02 & 0.12\\
	1/320 & 7.81 e-03 & 0.44\\
	\hline
 	\end{tabular}
 	\caption{ The averaged $L^2$ relative error in the last 1000 steps and the convergence rate for the 1D linear conservation law \eqref{eqn:linear} solved by the second-order scheme. A neural network with $4$ hidden layers and two shortcut connections is used and the batchsize is chosen as $10000$. The network width is set to be $20$ and the total number of parameters is $1341$.}
  \label{tbl:linear cv law high}
 \end{table}

 \begin{table}[H]
 	\centering 
 	\begin{tabular}{|c|c|c|}
 	\hline
 	$h = \sqrt{\Delta t}$ & error & order \\
 	\hline
 	1/10 & 7.36 e-02 & \\
	1/20 & 1.45 e-02 & 2.34\\
	1/40 & 2.33 e-03 & 2.63\\
	1/80 & 6.31 e-04 & 1.88\\
	1/160& 2.11 e-04 & 1.57\\
	\hline
 	\end{tabular}
 	\caption{ The averaged $L^2$ relative error in the last 1000 steps and the convergence rate for the 1D linear conservation law \eqref{eqn:linear} solved by the second-order scheme. A neural network with $6$ hidden layers and two shortcut connections is used and the batchsize is chosen as $10000$. The network width is set to be $60$ and the total number of parameters is $12951$.}
  \label{tbl:linear cv law high_deep}
 \end{table}
Results of the first-order and second-order schemes are summarized in Figure \ref{fig:1D_conservation_law}. The second-order scheme always has a better accuracy than the first-order scheme. A DNN with more parameters reduces the DNN approximation error and thus the convergence rate can be obtained over a larger range of grid size.
\begin{figure}[ht]
\centering
	\includegraphics[width=0.55\linewidth]{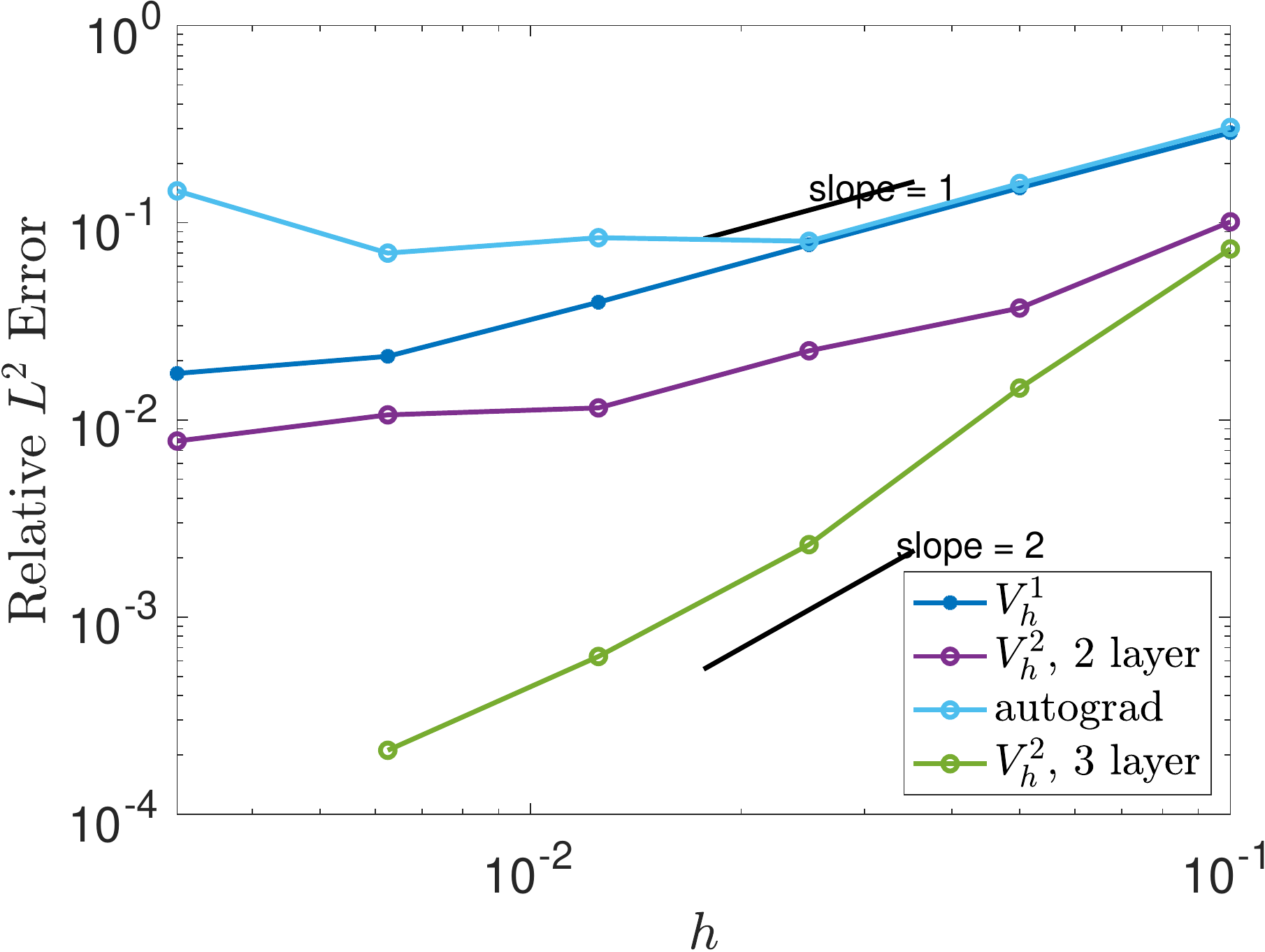}
	\caption{The averaged error of two schemes for the 1D conservation law with respect to the mesh size $h$. $V^1_h$ represents the first-order scheme with the forward Euler method in time and $V^2_h$ represents the second-order scheme with two DNNs and the forward Euler method in time. Autograd represents the first-order scheme with AutoGrad in time.}
	\label{fig:1D_conservation_law}
 \end{figure}

 \subsection{Burgers' equation}
Consider the Burgers' equation
 \begin{equation}\label{eqn:bur}
 	u_t + (\frac{u^2}{2})_x = 0
 \end{equation} 
 with initial condition 
 \begin{equation}\label{eqn:bur ic}
 	u(0,x) = \left\{
 	\begin{matrix}
 		1 & x<0\\
 		 0 & x>0
 	\end{matrix}
 	\right.
 \end{equation}
and reflecting boundary condition. The exact solution is discontinuous.
The numerical solution is constructed as 
 \begin{equation}\label{eqn:dnn-bur}
 \begin{aligned}
 	u_{\theta}(t,x) = \left\{\begin{matrix}
 		&\mathcal{N}_\theta(t,x_{i+\frac{1}{2}}) \varphi_i(x) &  t > 0 \quad x\in [x_i,x_{i+1})\\
 		&u(0,x_{i+\frac{1}{2}}) &  t = 0 \\
 	\end{matrix}\right.,
\end{aligned}
 \end{equation}
where $\varphi_i(x)$ is defined in \eqref{dnn-0}. This means that the numerical solution is approximated by a DNN at any point when $t>0$ and uses the exact solution when $t=0$. We divide the time interval $(0,T)$ into grids and assume that the temporal step size equals the spatial mesh size for simplicity. 
 
The loss function for the semi-discrete scheme is 
\begin{multline}\label{equ:loss bursemi}
 	\mathcal{L}_{\mathrm{semi}}(\theta) = \Bigg( \Delta t h\sum_{i,j} \bigg( \frac{\partial u_{\theta}(t_j,x_{i+\frac{1}{2}})}{\partial t} - \hat{f}^{\mathrm{God}}(u_{\theta}(t_j,x^-_i),u_{\theta}(t_j,x^+_i)) \\
 	+\hat{f}^{\mathrm{God}}(u_{\theta}(t_j,x^-_{i+1}),u_{\theta}(t_j,x^+_{i+1}))  \bigg)^2\Bigg)^{1/2},
\end{multline}
and the loss function for the fully-discrete scheme using the forward Euler method is
\begin{multline}\label{equ:loss burFE}
\mathcal{L}_{\mathrm{FE}}(\theta) = \Bigg(\sum_{i,j} \bigg( \frac{u_{\theta}(t_{j+1},x_{i+\frac{1}{2}})-u_{\theta}(t_j,x_{i+\frac{1}{2}})}{\Delta t} - \hat{f}^{\mathrm{God}}(u_{\theta}(t_j,x^-_i),u_{\theta}(t_j,x^+_i)) \\
+\hat{f}^{\mathrm{God}}(u_{\theta}(t_j,x^-_{i+1}),u_{\theta}(t_j,x^+_{i+1}))  \bigg)^2\Bigg)^{1/2},
\end{multline}
respectively. The error of these two loss functions is shown in Table \ref{tbl:burs error}. It is observed that the forward Euler method produces much better results than the autograd method for the Burgers' equation \eqref{eqn:bur} with a non-smooth solution \eqref{eqn:bur ic}. The detailed solution profiles are visualized in Figure \ref{fig:landscape of Bgs' eq}.
 \begin{table}[H]
 \centering
 \begin{tabular}{|c|c|c|c|c|}
 \hline
     $h=\Delta t$ & Fully discrete     & Semi-discrete \\
 	\hline
	1/10 & 9.87 e-02  &  3.82 e-01\\
    1/20 & 4.88 e-02  &  3.37 e-01\\
    1/40 & 3.48 e-02  &  3.03 e-01\\
	1/80 & 2.58 e-02  &  3.12 e-01\\
	1/160& 1.84 e-02  &  1.91 e-01\\
	1/320& 1.73 e-02  &  3.86 e-01\\
\hline
 \end{tabular}
 \caption{The averaged $L^2$ relative error in the last 1000 steps and the convergence rate for the Burgers' equation \eqref{eqn:bur}-\eqref{eqn:bur ic} with two loss functions \eqref{equ:loss burFE} and \eqref{equ:loss bursemi}. A neural network with $4$ hidden layers and two shortcut connections is used and the batchsize is chosen as $10000$. In 1D, the network width is set to be $20$ and the total number of parameters is $1341$.}
 \label{tbl:burs error}	
 \end{table}
 
 \begin{figure}[H]
	\subfigure[$t = 0.25$]{
		\includegraphics[width=0.45\linewidth]{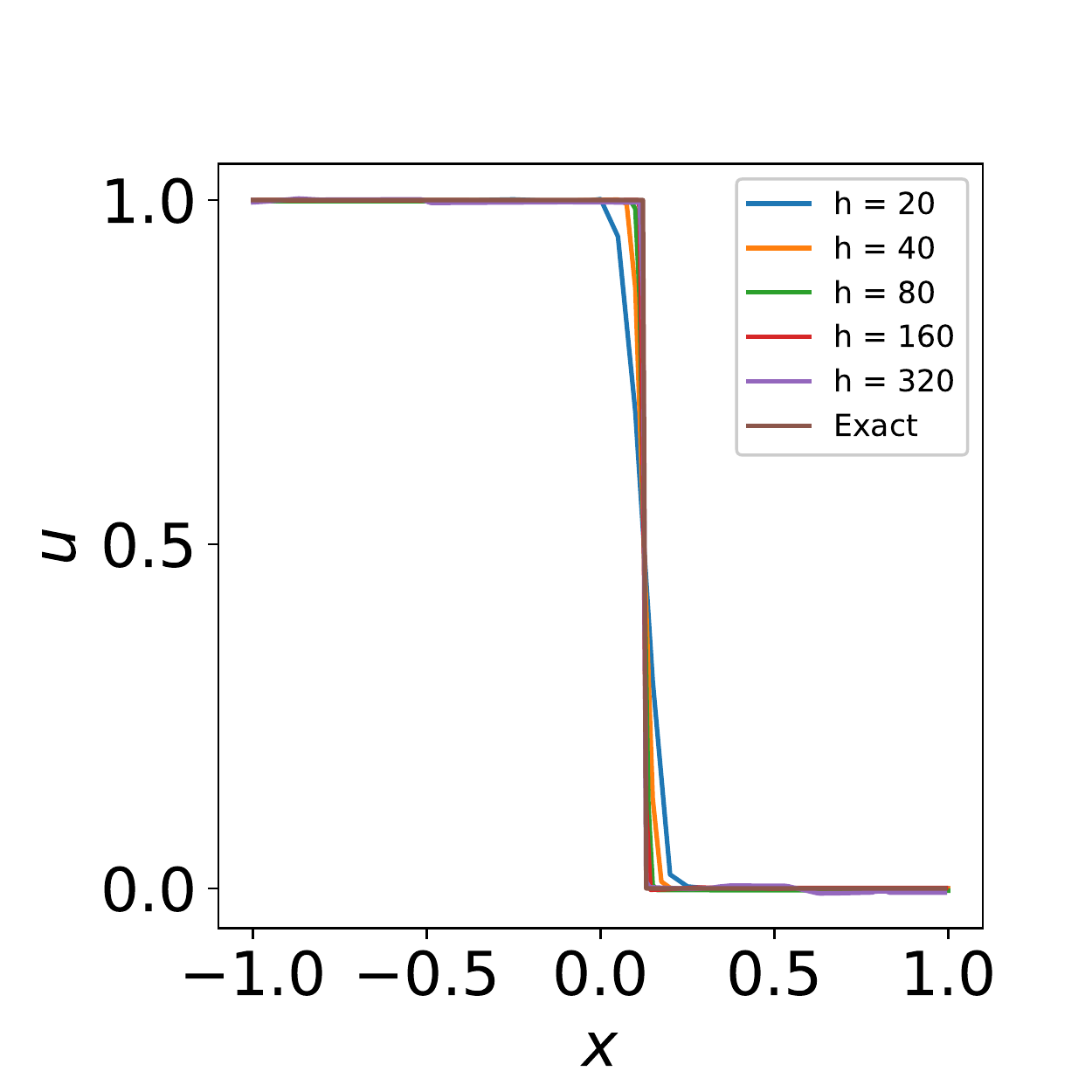}
	}
	\subfigure[$t = 0.5$]{
		\includegraphics[width=0.45\linewidth]{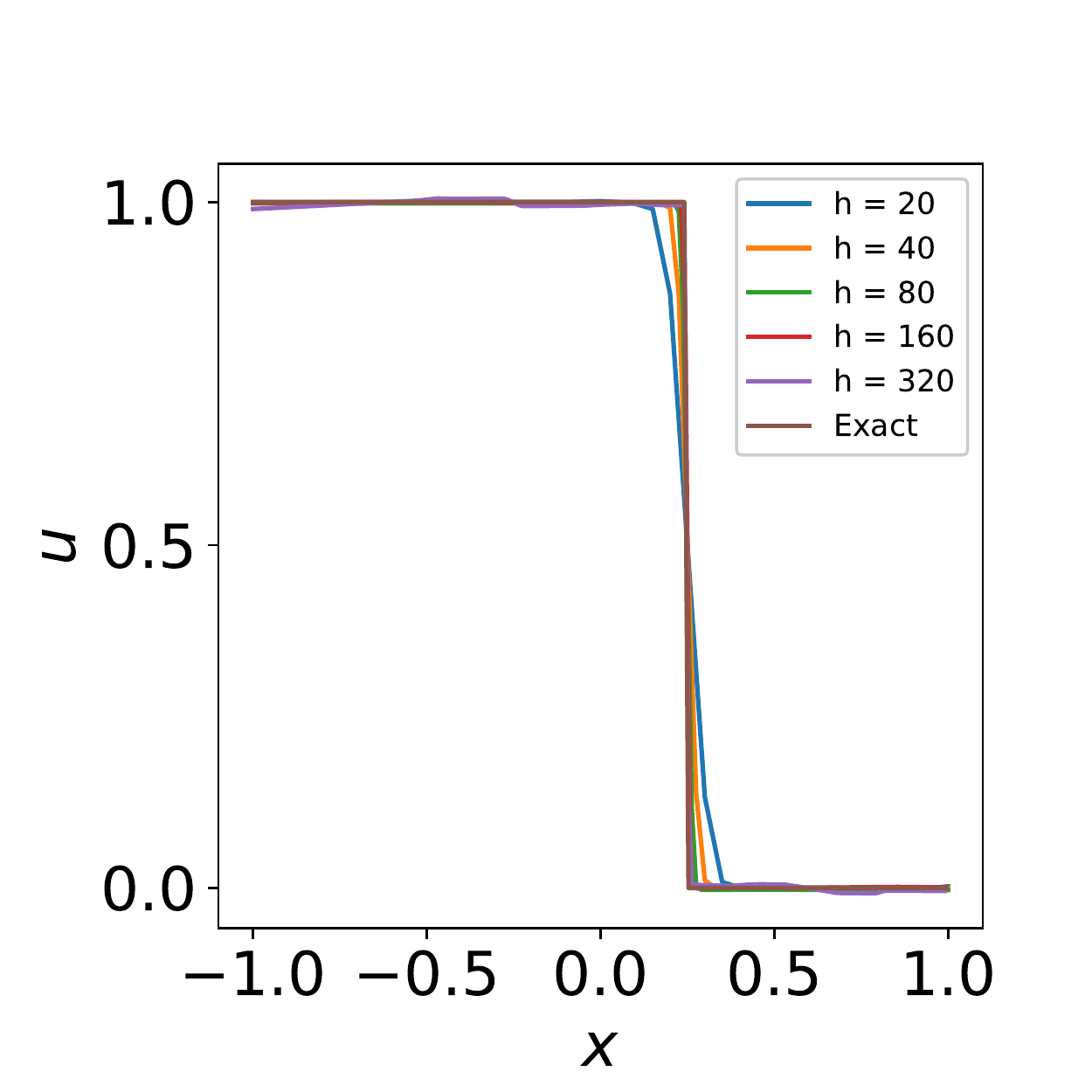}
	}
	
	\subfigure[$t = 0.75$]{
		\includegraphics[width=0.45\linewidth]{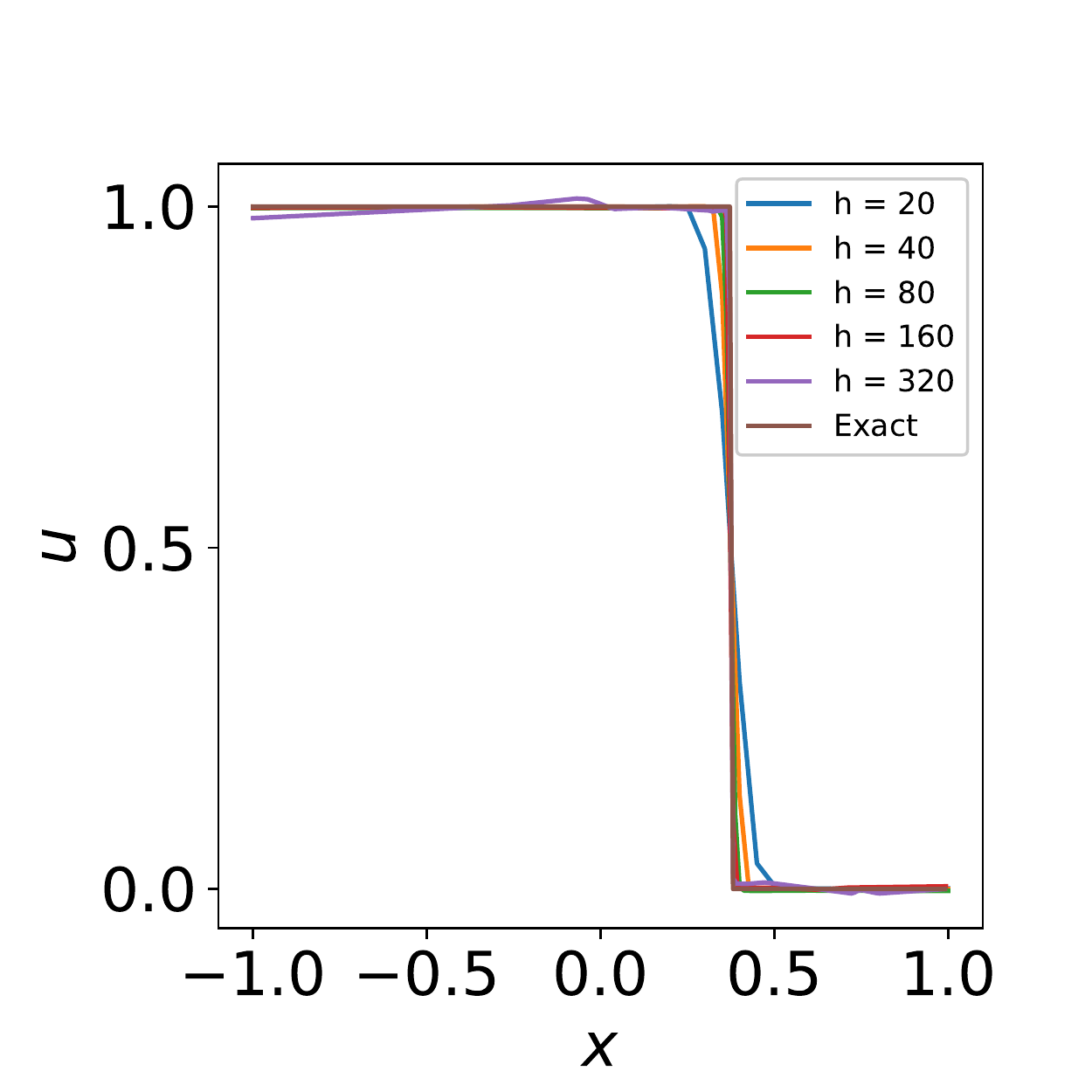}
	}
	\subfigure[$t = 1$]{
		\includegraphics[width=0.45\linewidth]{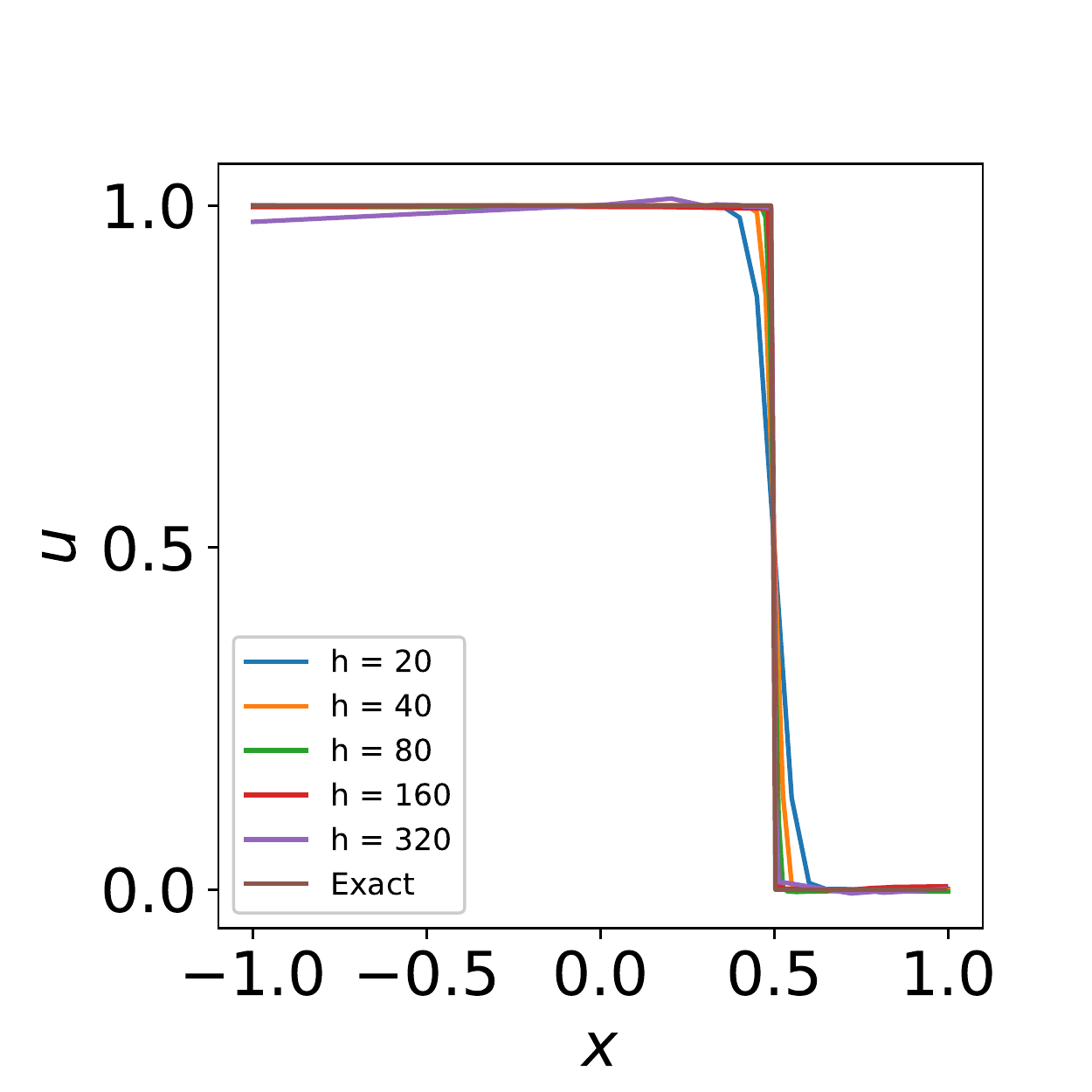}
	}
	\caption{1D solution profiles of the Burgers' equation \eqref{eqn:bur} approximated by the neural network solution \eqref{eqn:dnn-bur}.}
\label{fig:landscape of Bgs' eq}
\end{figure}

\subsection{Stochastic linear conservation law}
Consider the stochastic linear conservation law
\begin{equation}\label{eqn:linear3Dstochastic}
	\begin{aligned}
		2 d  \pi u_t  - (1+\exp(-\sum_{j=1}^s \omega_j)^2) \sum_{i=1}^d u_{x_i}  = 0 
	\end{aligned}
\end{equation}
with periodic boundary condition and initial condition $u(0,\boldsymbol{x},\boldsymbol{\omega}) = \sin\left(2\pi \sum_{i=1}^d x_i \right) $. The exact solution of the problem is $u(t,\boldsymbol{x},\boldsymbol{\omega}) = \sin\left((1+\exp(-\sum_{j=1}^s \omega_j)^2) t + 2\pi \sum_{i=1}^d x_i\right)$. 
The DNN solution is constructed as
\begin{equation}
\label{equ:construction}
\begin{aligned}
	&u(t,\boldsymbol{x},\boldsymbol{\omega}) = \sum_{\boldsymbol{i}} [t\mathcal{N}_\theta(t,\boldsymbol{x}_{\boldsymbol{i}+\frac{1}{2}},\boldsymbol{\omega}) + g(\boldsymbol{x}_{\boldsymbol{i}+\frac{1}{2}} )]\varphi_{\boldsymbol{i}}(\boldsymbol{x}),
\end{aligned}
\end{equation}
where $\varphi_{\boldsymbol{i}}(\boldsymbol{x})$ is defined in \eqref{equ:coeff}. Since the fully discrete scheme works better than the semi-discrete scheme, we only use the fully discrete scheme with the forward Euler method in time. The corresponding loss function reads as
\begin{multline}
  	 	\mathcal{L}_{\mathrm{FE}}(\theta) = \Bigg(\Delta t h^d\sum_{i,j}\bigg(2\pi \frac{ u_{\theta}(t_{j+1},\boldsymbol{x}_{\boldsymbol{i}+\frac12},\boldsymbol{\omega}) - u_{\theta}(t_j,\boldsymbol{x}_{\boldsymbol{i}+\frac12},\boldsymbol{\omega})}{\Delta t} \\
		- (1+\exp(-\sum_{j=1}^s \omega_j)^2) \frac{u_{\theta}(t_j,\boldsymbol{x}_{\boldsymbol{i}+1},\bom) -u_{\theta}(t_j,\boldsymbol{x}_{\boldsymbol{i}},\boldsymbol{\omega})}{h} \bigg)^2\Bigg)^{1/2}.
 \end{multline}
Figure \ref{fig:landscape of CV eq} plots the expectation and the variance of the solution along the line $x_1=x_2=x_3$ when $d=3$, $s=2$, and $s=5$. Table \ref{tbl:cv error stochastic} records the relative $L^2$ errors of the expectation and the variance for $s = 50$ and $100$ respectively. The first-order accuracy is observed for the stochastic linear conservation law in both expectation and variance.
\begin{figure}[htbp]
\centering
	\subfigure[$s=2$]{
		\includegraphics[width=0.45\linewidth]{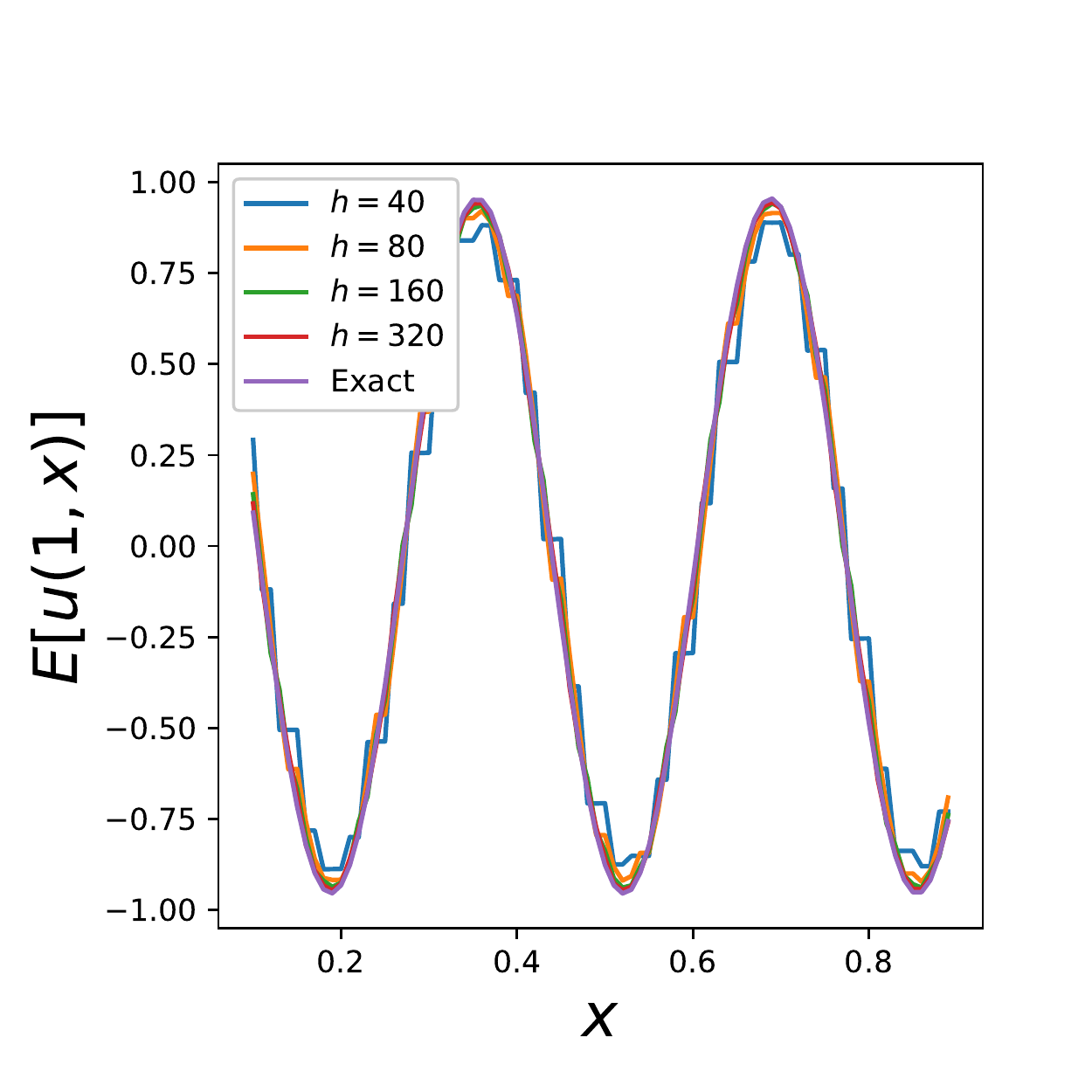}
	}
	\subfigure[$s=2$]{
		\includegraphics[width=0.45\linewidth]{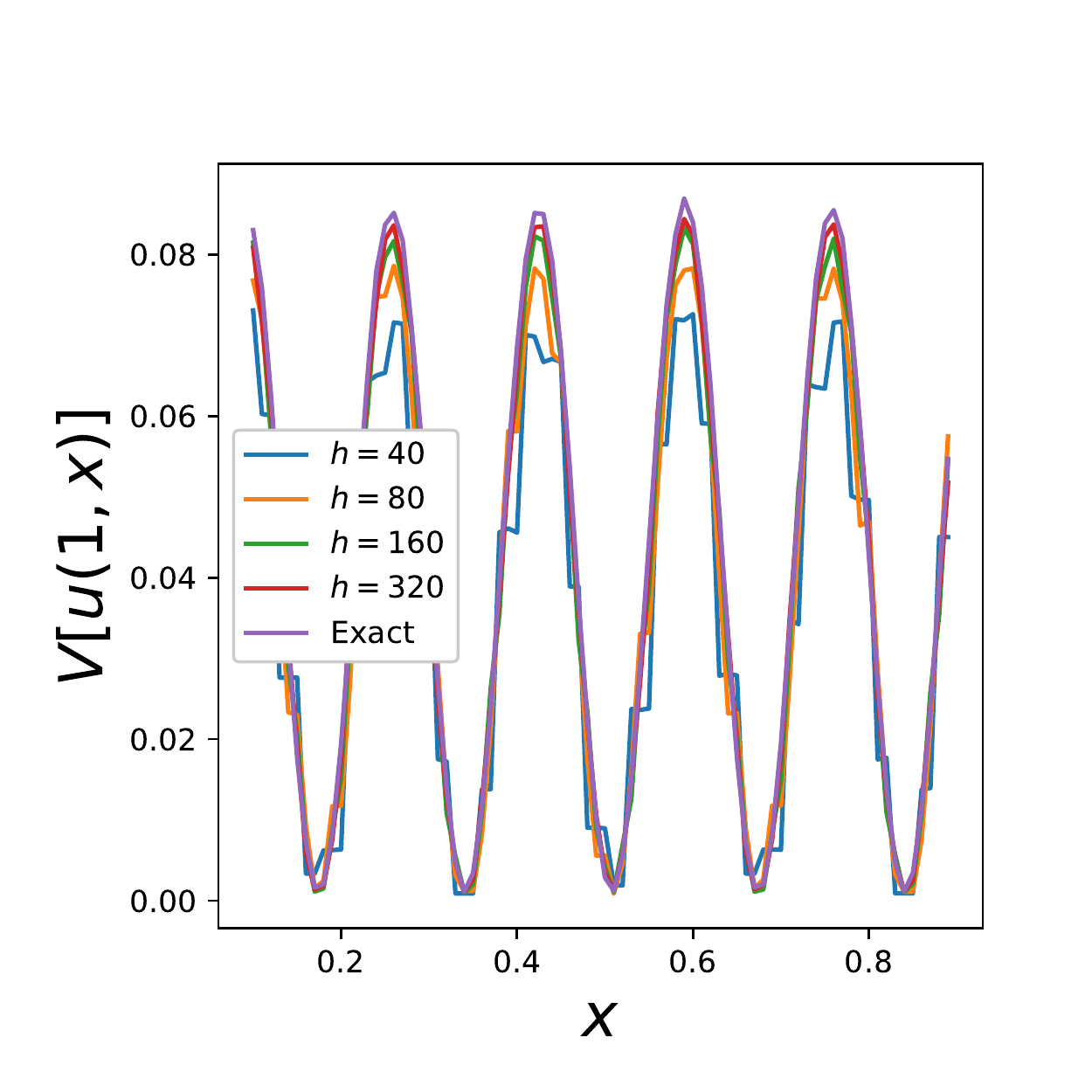}
	}
	
	\subfigure[$s=5$]{
		\includegraphics[width=0.45\linewidth]{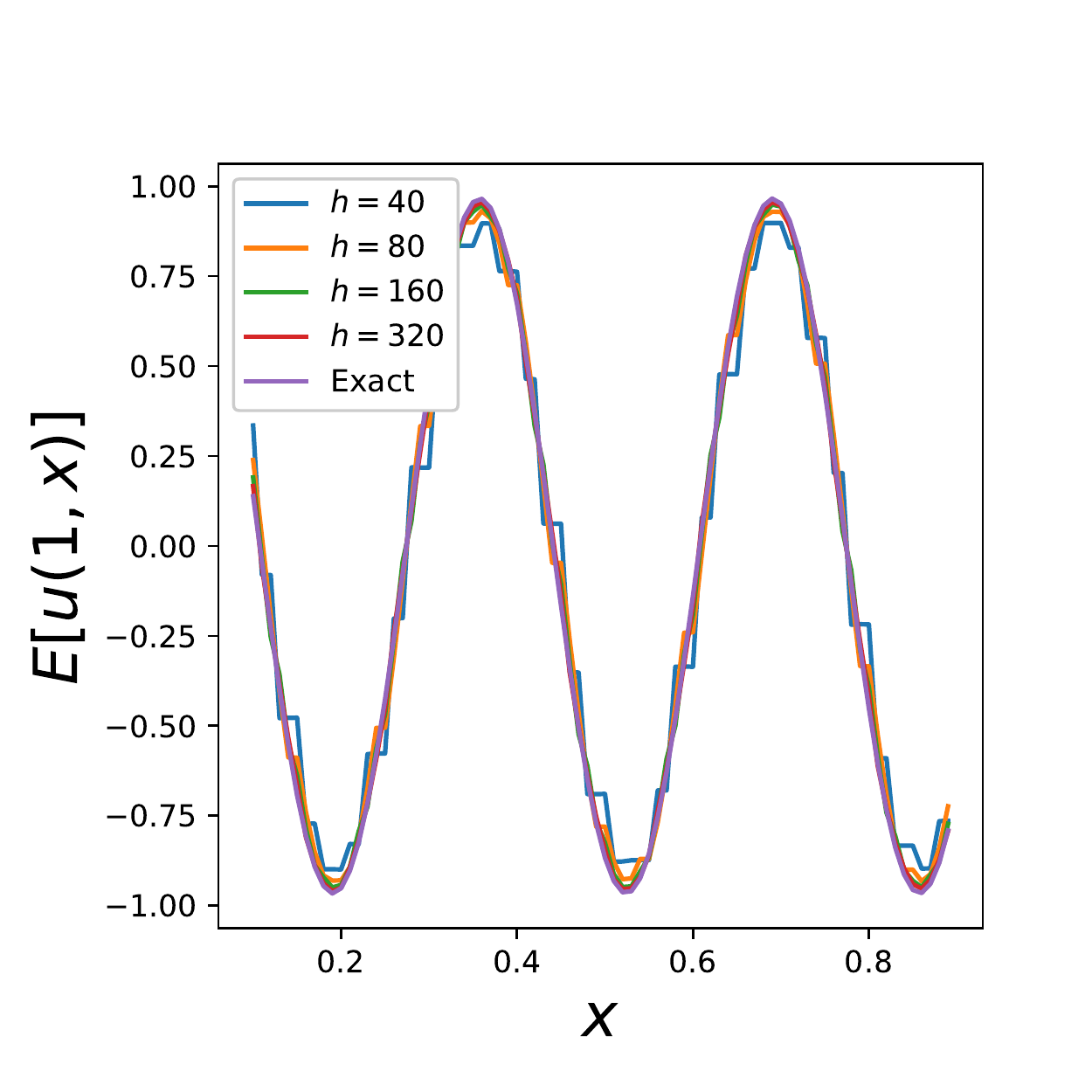}
	}
	\subfigure[$s=5$]{
		\includegraphics[width=0.45\linewidth]{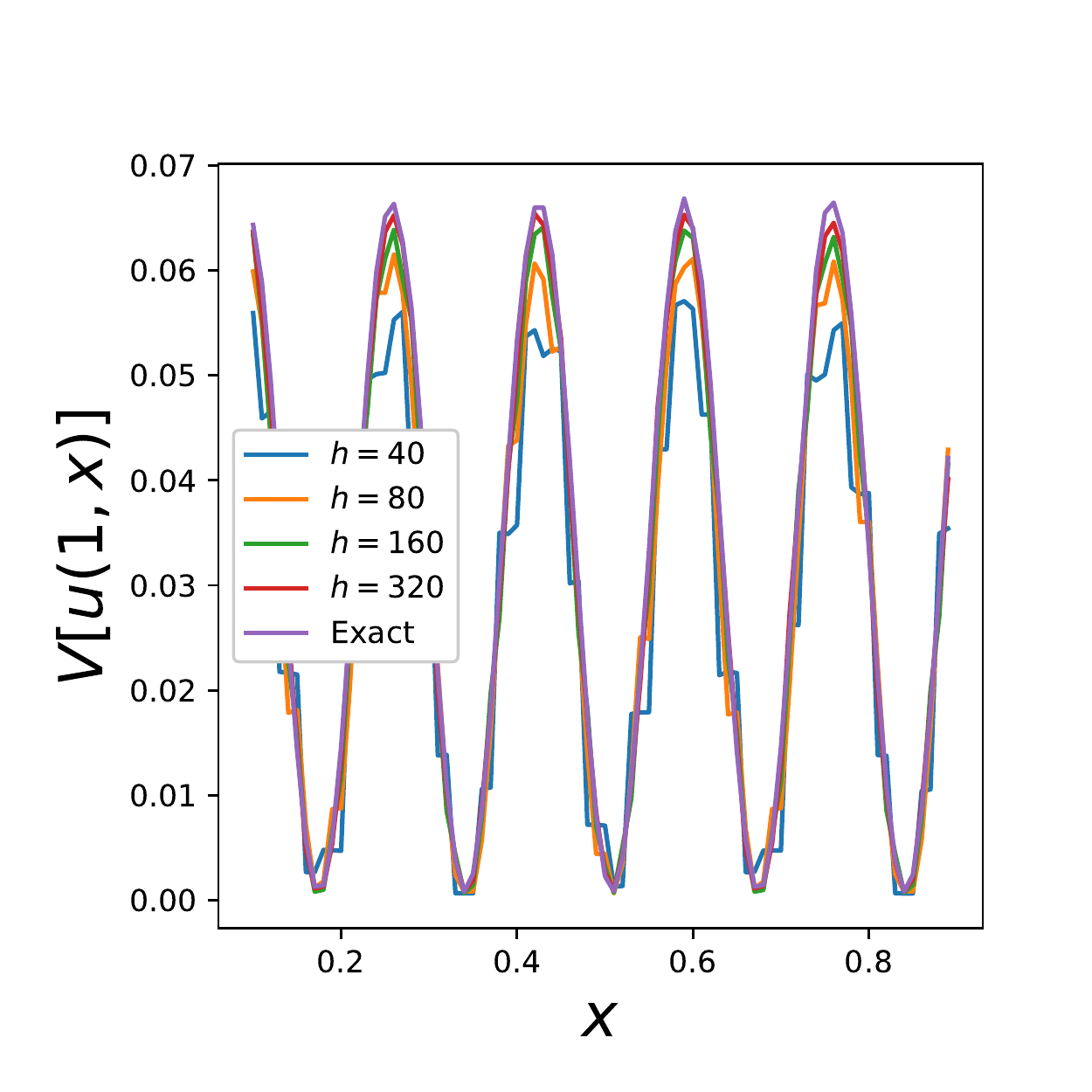}
	}
	\caption{Numerical and exact expectations and variances of the solution to the stochastic linear conservation law \eqref{eqn:linear3Dstochastic} along the line where $x_1=x_2=x_3$ when $d=3$, $s=2$ and $s=5$. When $s=2$, the network width is $40$ and the total number of parameters is $8441$. When $s=5$, the network width is $50$ and the total number of parameters is $13211$. The batchsize is 200000.}
\label{fig:landscape of CV eq}
\end{figure}
\begin{table}[ht]
 \centering
 \begin{tabular}{|c|c|c|c|c|c|c|}
 \hline
     $s$ & $h = \Delta t$ &  Expectation & Order & Variance  & Order\\
 	\hline
     50 & 1/40 & 1.54 e-01 & & 2.13 e-01 &  \\
     50 & 1/80 & 7.85 e-02 & 0.97 & 1.14 e-01 & 0.93 \\
     50 & 1/160 & 3.88 e-02 & 1.01 & 5.61 e-02 & 0.96\\
     50 & 1/320 & 1.96 e-02 & 0.98 & 3.22 e-02 & 0.79\\
    \hline
	 100 & 1/40 & 1.53 e-01 &  & 2.07 e-01 & \\
	 100 & 1/80 & 7.83 e-02 & 0.97 &  1.12 e-01 & 0.88\\
	 100 & 1/160 & 3.93 e-02 & 0.99 & 5.82 e-02 & 0.95\\
	 100 & 1/320 & 2.01 e-02 & 0.96 &  2.93 e-02 & 0.98\\
\hline
 \end{tabular}
\caption{The averaged $L^2$ relative error in the last 1000 steps and the convergence rate for the stochastic conservation law \eqref{eqn:linear3Dstochastic}. The neural network used here has 6 hidden layers and 3 shortcut connections. When $s=50$, the network width is $100$ and the total number of parameters is $56101$. When $s=100$, the network width is $200$ and the total number of parameters is $222201$. The batchsize is 200000.}
\label{tbl:cv error stochastic}	
\end{table}

\subsection{Stochastic Burgers' equation}
Consider the stochastic Burgers' equation defined as
 \begin{equation}\label{eqn:bur stochastic}
 	u_t + (\frac{u^2}{2})_x = 0
 \end{equation} 
 with initial condition 
 \begin{equation}\label{eqn:bur stochastic ic}
 	u(0,x,\boldsymbol{\omega}) = \left\{
 	\begin{matrix}
 		1+\epsilon \sum_{i=1}^s \omega_i & x<0 \\
 		 0 & x>0 
 	\end{matrix}
 	\right..
 \end{equation}
The exact solution is
\begin{equation}
	 	u(t,x,\boldsymbol{\omega}) = \left\{
 	\begin{matrix}
 		z  & x< \frac{z}{2}\\
 		 0 & x>\frac{z}{2}
 	\end{matrix}
 	\right.,
\end{equation}
where $z = 1+\eps \sum_{i=1}^s \om_i$. 
The expectation of the solution is
\begin{equation}
		 	\mathbb{E}_{\boldsymbol{\omega}}[u(t,x,\boldsymbol{\omega})] = \left\{
 	\begin{matrix}
 		1  & x< \frac{1-\eps}{2}\\
 		\frac{1-4x^2+2\eps+\eps^2}{4\eps} & \frac{1+\eps}{2}
>x> \frac{1-\eps}{2}\\
 		 0 & x>\frac{1+\eps}{2}
 	\end{matrix}\right..
\end{equation}
The reference variance of the solution is simulated  by the MC method.
The neural network setup for different $s$ is listed in Table \ref{tbl:Neural network setting for s Bur problem}.
\begin{table}
	\centering
	\begin{tabular}{|c|c|c|c|c|}
		\hline
		$s$ & number of hidden layers & network width & number of parameters\\
		\hline
		2  & 6 & 40 & 8441\\
		5  & 6 & 50 & 13211	\\
		10 & 6 & 50 & 13451\\
		50 & 6 & 100 & 55901\\
		100& 6 & 200 & 221801\\
		200& 6 & 400 & 883601\\
		\hline
	\end{tabular}
	\caption{Network setups for stochastic equations with different number of random variables.}
	\label{tbl:Neural network setting for s Bur problem}
\end{table}
The approximate solution is constructed as 
 \begin{equation}
 \begin{aligned}
 	u_{\theta}(t,x,\boldsymbol{\omega}) = \left\{\begin{matrix}
 		&\mathcal{N}_\theta(t,x_{i+\frac{1}{2}},\boldsymbol{\omega}) \varphi_i(x) &  t > 0 \;\; x\in(x_i,x_{i+1})\\
 		&u(0,x_{i+\frac{1}{2}},\boldsymbol{\omega}) &  t = 0 \\
 	\end{matrix}\right.,
\end{aligned}
 \end{equation}
and the loss function is the same as \eqref{equ:loss burFE}. Expectation and variance errors of the proposed method are recorded in Table \ref{tbl: s Bur equation 10000 eroor} and Table \ref{tbl: s Bur equation 50000 eroor} when the batch size is $10000$ and $50000$, respectively. The relative $L^2$ error in expectation and variance reduces when the batch size is increased and the relative $L^1$ error is slightly better than the $L^2$ error. Furthermore, we apply the  quasi-Monte Carlo method \cite{caflisch1998monte,chen2021quasi} to approximate the loss function; see Table \ref{tbl: s Bur equation 50000 quasi}. It is found that the error in this case is smaller than that of the MC method but cannot be further reduced with smaller mesh sizes. In addition, we apply the multilevel MC method \cite{gopalakrishnan2003multilevel} to approximate the loss function and the numerical result is recorded in Table \ref{tbl: s Bur equation multi level}. Again, slightly better results are obtained but the approximation of the variance is not good.
\begin{figure}
	\subfigure[$s=2, \epsilon = 0.25$]{\includegraphics[width=0.45\linewidth]{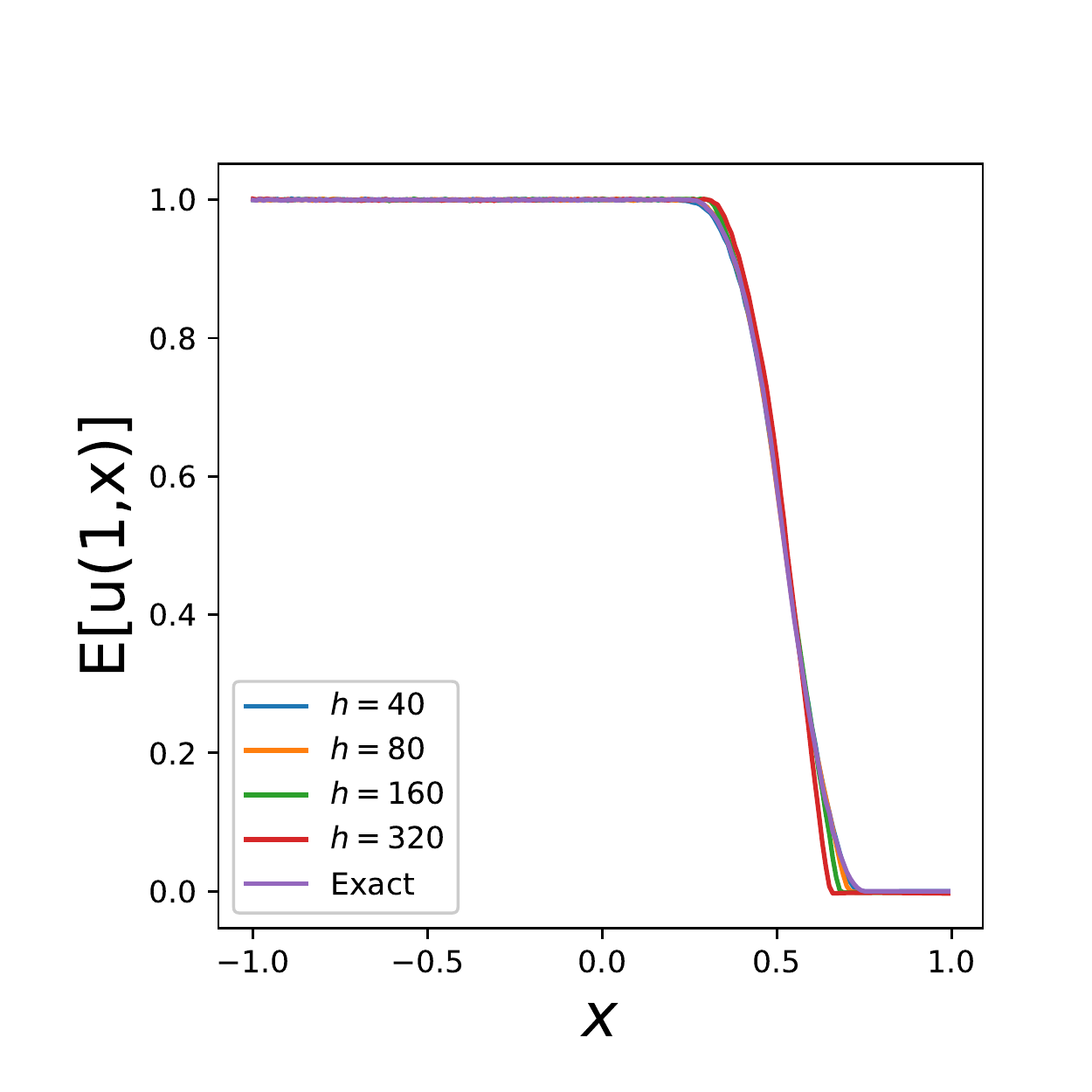}
	}
	\subfigure[$s=2, \epsilon = 0.25$]{\includegraphics[width=0.45\linewidth]{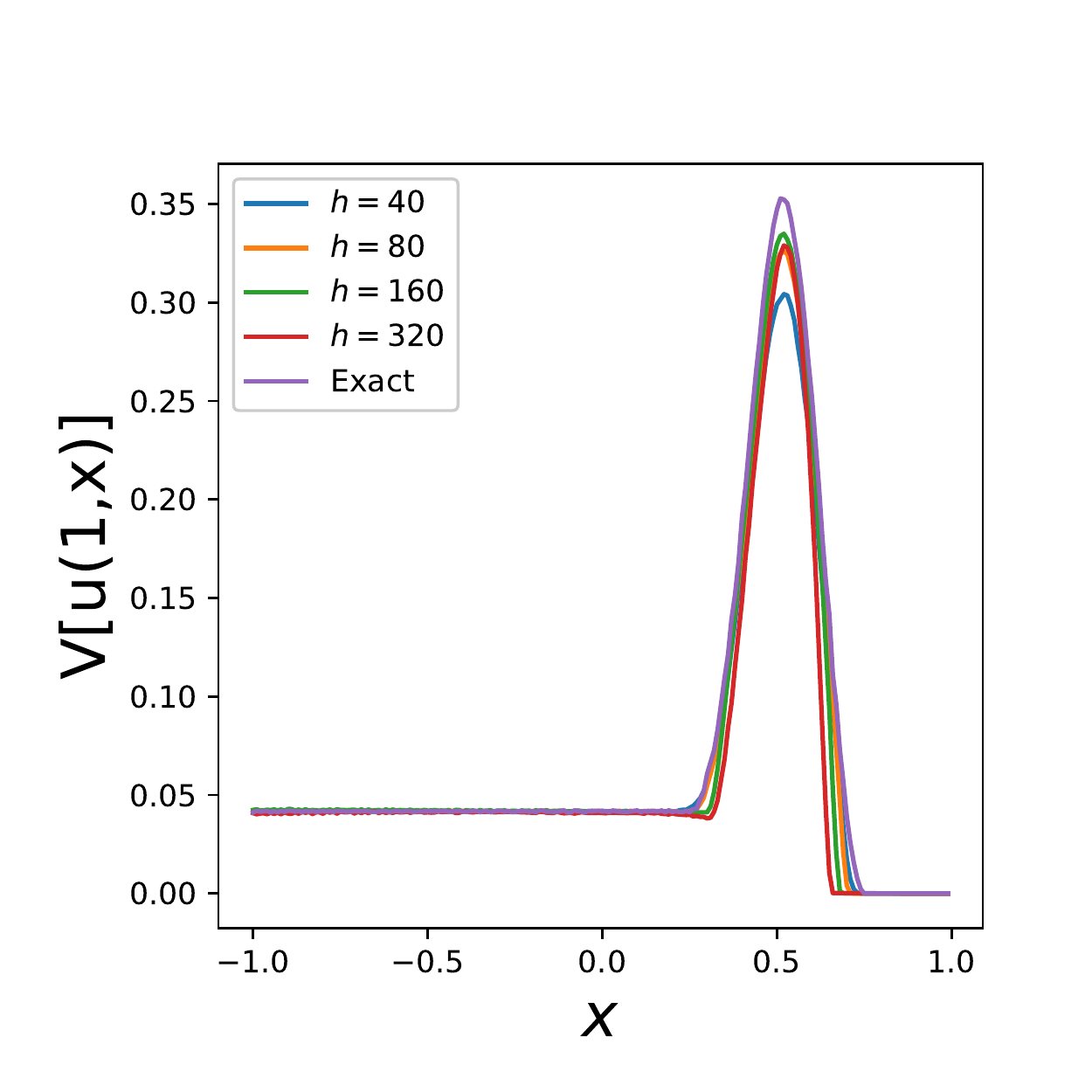}
	}
\newline
	\subfigure[$s=10, \epsilon = 0.05$]{
		\includegraphics[width=0.45\linewidth]{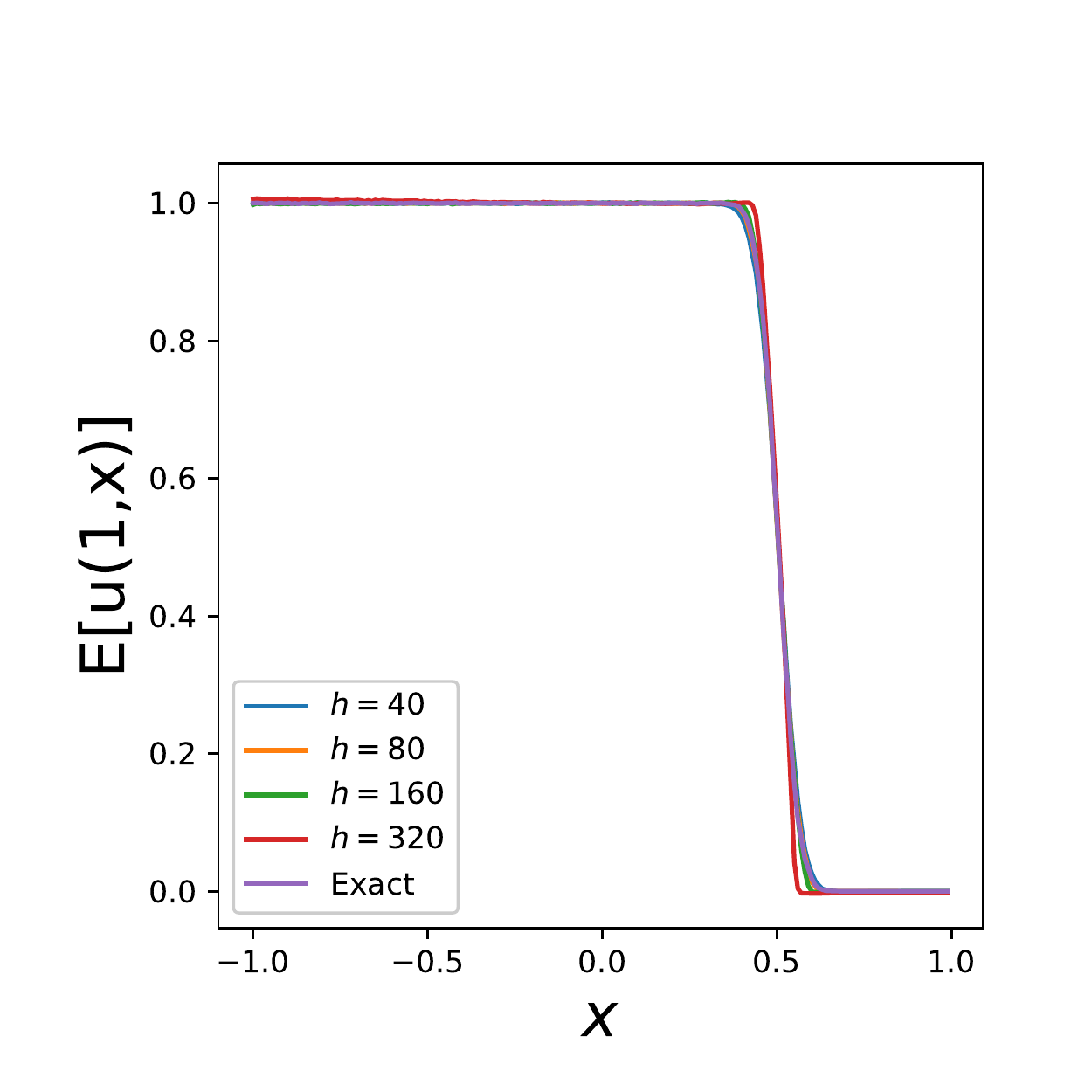}
	}
	\subfigure[$s=10, \epsilon = 0.05$]{
		\includegraphics[width=0.45\linewidth]{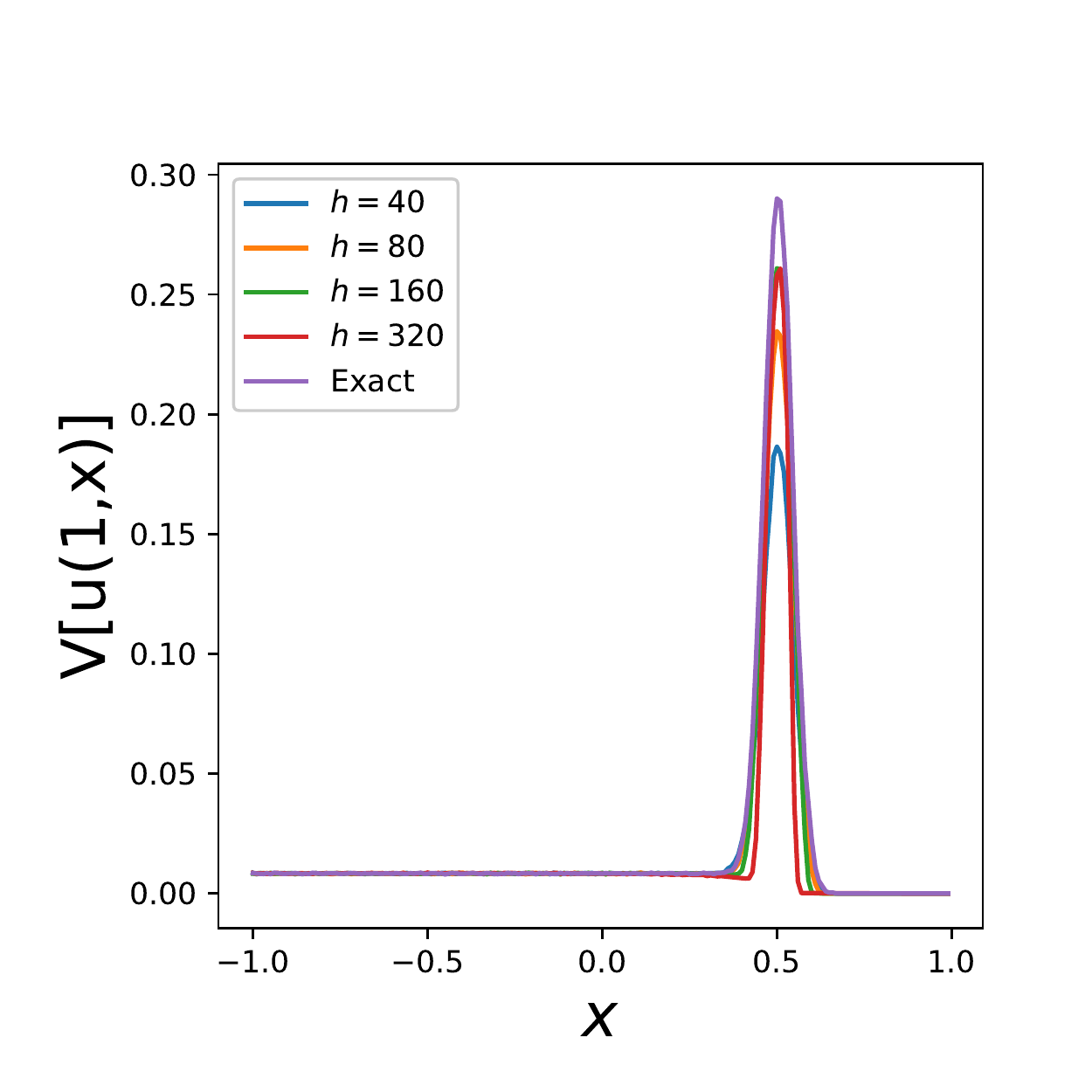}
	}
\newline
	\subfigure[$s=100, \epsilon = 0.005$]{
		\includegraphics[width=0.45\linewidth]{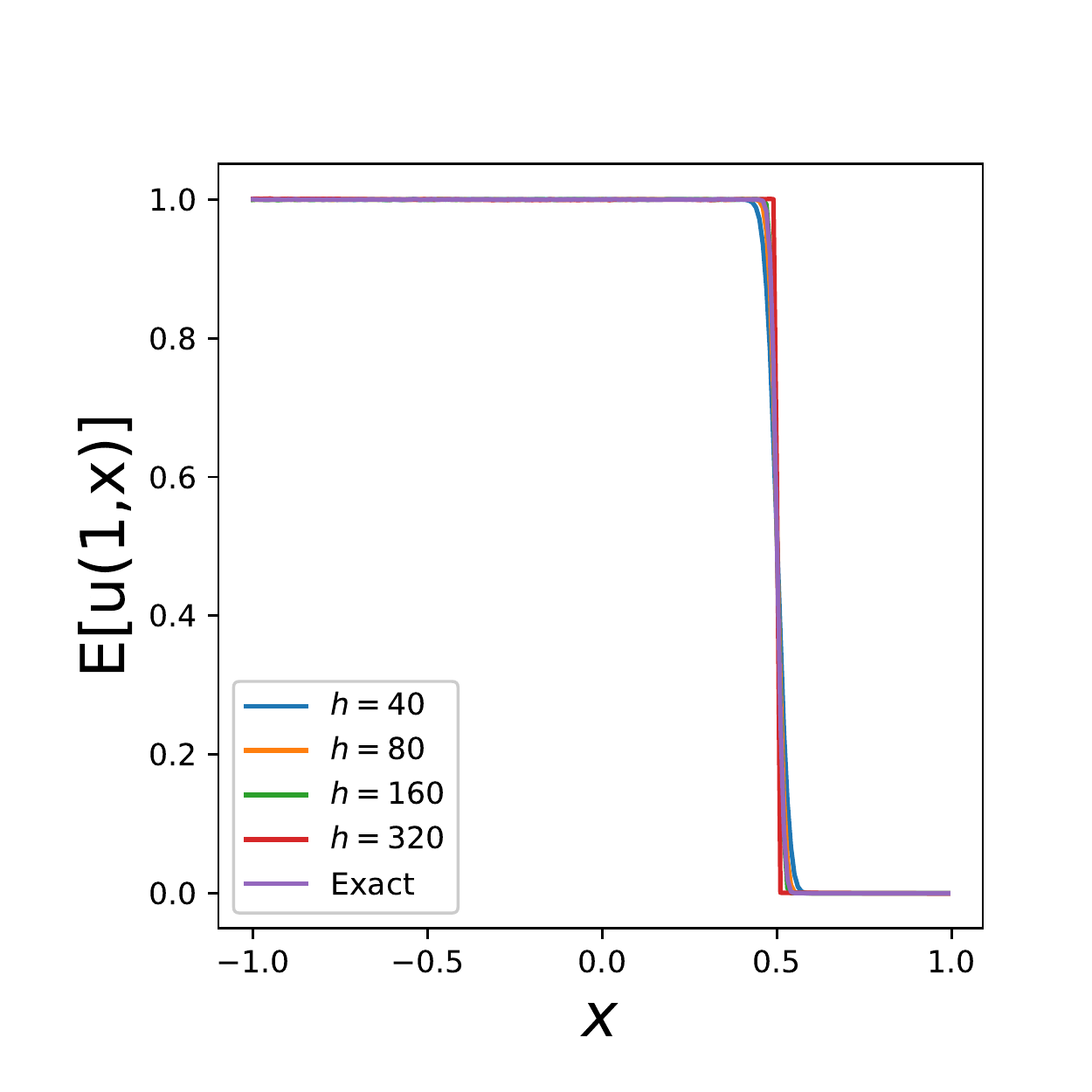}
	}
	\subfigure[$s=100, \epsilon = 0.005$]{
		\includegraphics[width=0.45\linewidth]{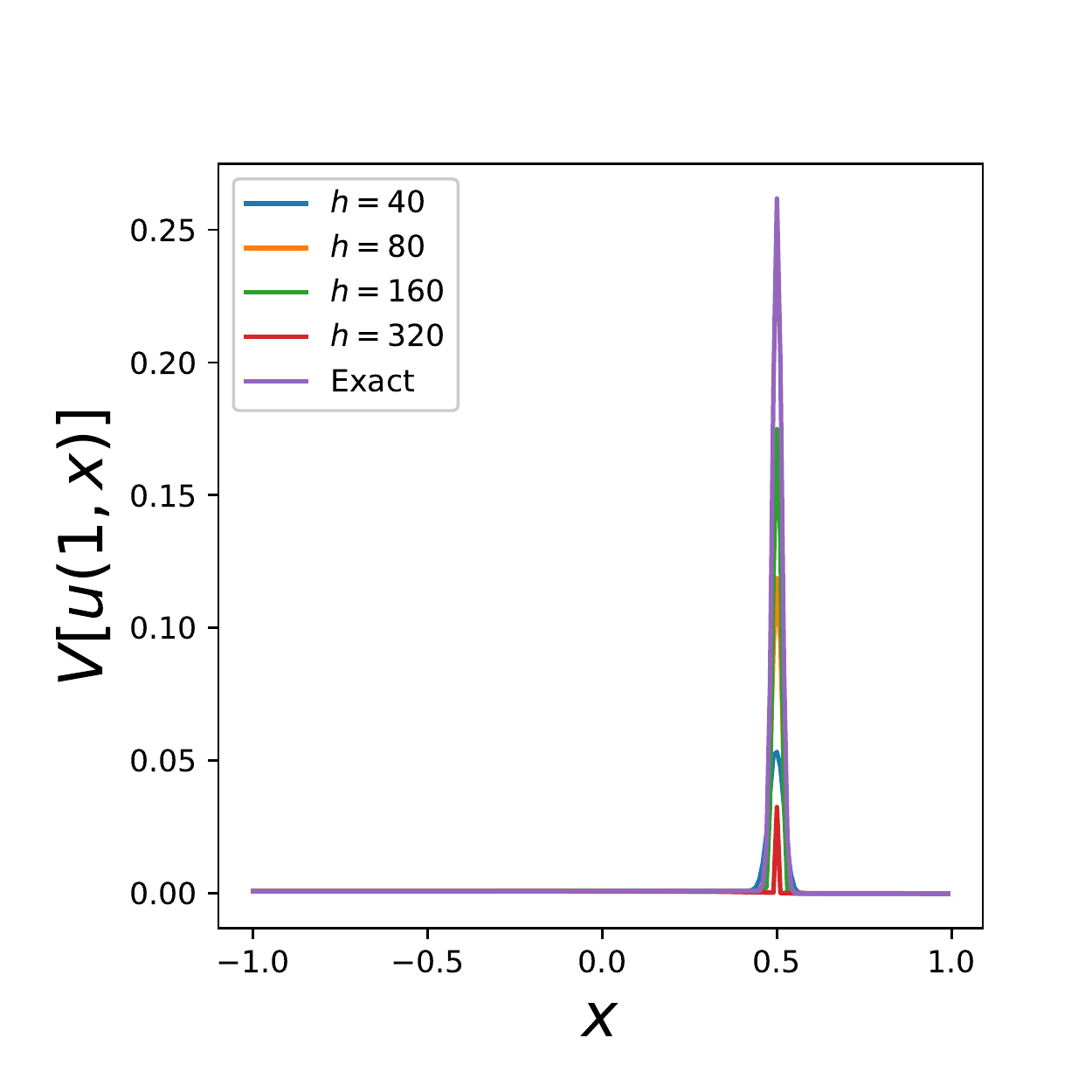}
	}
	\caption{1D solution profiles of the stochastic Burgers' equation.}
\label{fig:landscape of stochastic exceptation Bgs' eq 3}
\end{figure}
\begin{table}[htbp]
\centering
\begin{tabular}{|c|c|c|c|c|}
\hline
	$\epsilon$ & s & h & Expectation error ($L^2$) & Variance error ($L^2$)\\
		\hline
	0.25 & 2 & 1/40 &  1.00 e-2 & 5.42 e-1\\
	0.25 & 2 & 1/80 &  2.98 e-2 & 6.49 e-1\\
	0.1  & 5 & 1/40 &  1.48 e-2 & 2.23 e-1\\
	0.1  & 5 & 1/80 &  3.06 e-2 & 3.22 e-1\\
	0.05 & 10 & 1/40 & 8.16 e-3 & 2.75 e-1 \\
	0.05 & 10 & 1/80 & 2.24 e-2 & 4.34 e-1\\
	0.01 & 50 & 1/40 & 1.09 e-2 & 5.78 e-1\\
	0.01 & 50 & 1/80 & 1.90 e-2 & 5.86 e-1\\
	0.005 & 100 & 1/40 & 5.30 e-3 & 6.82 e-1 \\
	0.005 & 100 & 1/80 & 1.81 e-3 & 7.89 e-1\\
	0.0025 & 200 & 1/40 & 1.02 e-2 & 8.96 e-1 \\
	0.0025 & 200 & 1/80 & 1.57 e-2 & 9.92 e-1\\
	\hline
\end{tabular}
\caption{Expectation and variance errors of the proposed method for the stochastic Burgers' equation when the MC method is used with the batchsize $10000$.}
\label{tbl: s Bur equation 10000 eroor}
\end{table}
\begin{table}[htbp]
\centering
\begin{tabular}{|c|c|c|c|c|c|c|}
\hline
     \multirow{2}*{$\eps$} & \multirow{2}*{$s$}  &\multirow{2}*{$h$}   & \multicolumn{2}{c|}{$L^2$ error}            & \multicolumn{2}{c|}{$L^1$ error}  \\
     \cline{4-7}
     	~ &~ & ~ & Expectation & Variance & Expectation & Variance \\
		\hline
	0.25 & 2 & 1/40  &  2.44 e-3 & 1.27 e-1 & 1.63 e-03 & 8.66 e-02\\
	0.25 & 2 & 1/80  &  3.57 e-3 & 8.82 e-2 & 2.06 e-03 & 5.77 e-02\\
	0.25 & 2 & 1/160 &  9.70 e-3 & 1.14 e-1 & 4.42 e-03 & 7.40 e-02\\
	0.25 & 2 & 1/320 &  2.28 e-2 & 2.21 e-1 & 1.10 e-02 & 1.46 e-01\\
	0.1  & 5 & 1/40  &  4.16 e-3 & 2.30 e-1 & 2.03 e-03 & 1.60 e-01\\
	0.1  & 5 & 1/80  &  2.44 e-3 & 1.24 e-1 & 1.34 e-03 & 9.78 e-02\\
	0.1  & 5 & 1/160 &  4.34 e-3 & 9.16 e-2 & 2.17 e-03 & 8.04 e-02\\
	0.1  & 5 & 1/320 &  1.61 e-2 & 2.21 e-1 & 8.10 e-03 & 1.75 e-01\\
	0.05 & 10 & 1/40 & 6.79 e-3 & 3.37 e-1  & 2.99 e-03 & 2.40 e-01\\
	0.05 & 10 & 1/80 & 2.25 e-3 & 1.86 e-1  & 1.13 e-03 & 1.45 e-01\\
	0.05 & 10 & 1/160 & 4.68 e-3 & 1.27 e-1 & 2.28 e-03 & 1.17 e-01 \\
	0.05 & 10 & 1/320 & 2.01 e-2 & 3.36 e-1 & 8.94 e-03 & 2.74 e-01\\
	0.01 & 50 & 1/40 & 1.80 e-2 & 6.42 e-1  & 5.36 e-03 & 5.01 e-01\\
	0.01 & 50 & 1/80 & 5.74 e-3 & 4.04 e-1 & 1.67 e-03 & 3.32 e-01\\
	0.01 & 50 & 1/160 & 3.09 e-3 & 2.69 e-1 & 1.18 e-03 & 2.68 e-01\\
	0.01 & 50 & 1/320 & 4.40 e-2 & 9.12 e-1 & 1.70 e-02 & 8.06 e-01\\
	0.005 & 100 & 1/40 & 2.58 e-2 & 7.53 e-1 & 6.54 e-03 & 5.01 e-01\\
	0.005 & 100 & 1/80 & 8.64 e-3 & 5.25 e-1 & 2.09 e-03 & 3.32 e-01\\
	0.005 & 100 & 1/160 & 1.95 e-3 & 3.76 e-1& 7.22 e-04 & 2.68 e-01\\
	0.005 & 100 & 1/320 & 3.07 e-2 & 9.47 e-1& 5.65 e-03 & 8.06 e-01\\
	0.0025 & 200 & 1/40 & 2.58 e-2 & 7.53 e-1 & 7.56 e-03 & 7.52 e-01\\
	0.0025 & 200 & 1/80 & 8.64 e-3 & 5.25 e-1 & 2.51 e-03 & 5.68 e-01\\
	0.0025 & 200 & 1/160 & 1.95 e-3 & 3.76 e-1 & 8.30 e-04 & 5.51 e-01\\
	0.0025 & 200 & 1/320 & 3.07 e-2 & 9.47 e-1 & 3.52 e-03 & 9.09 e-01\\
	\hline
\end{tabular}
\caption{Expectation and variance errors of the proposed method for the stochastic Burgers' equation when the MC method is used with the batchsize $50000$.}
\label{tbl: s Bur equation 50000 eroor}
\end{table}
\begin{table}[htbp]
\centering
\begin{tabular}{|c|c|c|c|c|c|c|}
\hline
     \multirow{2}*{$\eps$} & \multirow{2}*{$s$}  &\multirow{2}*{$h$}   & \multicolumn{2}{c|}{$L^2$ error}            & \multicolumn{2}{c|}{$L^1$ error}  \\
     \cline{4-7}
     	~ &~ & ~ & Expectation & Variance & Expectation & Variance \\
		\hline
	0.25 & 2 & 1/80  &  3.88 e-3 & 8.06 e-2 & 2.17 e-03 & 8.06 e-02\\
	0.25 & 2 & 1/160 &  7.78 e-3 & 8.66 e-3 & 4.48 e-03 & 8.66 e-02\\
	0.25 & 2 & 1/320 &  2.68 e-2 & 2.44 e-1 & 1.72 e-02 & 2.44 e-01\\
	0.1  & 5 & 1/80  &  2.39 e-3 & 1.21 e-1 & 1.35 e-03 & 8.78 e-02\\
	0.1  & 5 & 1/160 &  3.18 e-3 & 7.39 e-2 & 2.24 e-03 & 6.29 e-02\\
	0.1  & 5 & 1/320 &  1.70 e-2 & 2.42 e-1 & 9.61 e-03 & 1.92 e-01\\
	0.05 & 10 & 1/80 &  1.80 e-3 & 1.88 e-1 & 1.14 e-03 & 1.45 e-01\\
	0.05 & 10 & 1/160 &  4.23 e-3 & 1.30 e-1 & 2.32 e-03 & 1.22 e-01 \\
	0.05 & 10 & 1/320 &  1.47 e-2 & 2.56 e-1 & 7.40 e-03 & 2.27 e-01\\
	0.01 & 50 & 1/80 & 5.88 e-3 & 4.01 e-1 & 1.79 e-03 & 3,25 e-01\\
	0.01 & 50 & 1/160 & 3.61 e-3 & 2.74 e-1 & 1.76 e-03 & 2.74 e-01\\
	0.01 & 50 & 1/320 & 2.19 e-2 & 5.37 e-1 & 5.47 e-03 & 5.04 e-01\\
	0.005 & 100 & 1/80 & 8.79 e-3 & 5.29 e-1 & 2.11 e-03 & 4.43 e-01\\
	0.005 & 100 & 1/160 & 2.78 e-3 & 3.47 e-1 & 1.21 e-03 & 3.49 e-01\\
	0.005 & 100 & 1/320 & 3.01 e-2 & 8.97 e-1 & 7.03 e-03 & 8.20 e-01\\
	\hline
\end{tabular}
\caption{Expectation and variance errors of the proposed method for the stochastic Burgers' equation when the quasi-Monte Carlo method is used with the batchsize $50000$.}
\label{tbl: s Bur equation 50000 quasi}
\end{table}
\begin{table}[htbp]
\centering
\begin{tabular}{|c|c|c|c|c|c|c|}
\hline
     \multirow{2}*{$\eps$} & \multirow{2}*{$s$}  &\multirow{2}*{$h$}   & \multicolumn{2}{c|}{$L^2$ error}            & \multicolumn{2}{c|}{$L^1$ error}  \\
     \cline{4-7}
     	~ &~ & ~ & Expectation & Variance & Expectation & Variance\\
		\hline
	0.01 & 50 & 1/80 & 5.82 e-3 & 4.03 e-1 & 2.11 e-03 & 3.27 e-01\\
	0.01 & 50 & 1/160 & 2.37 e-3 & 2.36 e-1 & 9.99 e-04 & 2.18 e-01\\
	0.01 & 50 & 1/320 & 8.37 e-3 & 2.76 e-1 & 3.42 e-03 & 2.80 e-01\\
	0.005 & 100 & 1/80 & 8.82 e-3 & 5.28 e-1 &	2.20 e-03 & 4.41 e-01\\
	0.005 & 100 & 1/160 & 2.07 e-3 & 3.53 e-1 & 7.66 e-04 & 3.51 e-01\\
	0.005 & 100 & 1/320 & 3.03 e-2 & 8.72 e-1 & 5.92 e-03 & 7.96 e-01\\
	\hline
\end{tabular}
\caption{Expectation and variance errors of the proposed method for the stochastic Burgers' equation when the multi-level MC method is used.}
\label{tbl: s Bur equation multi level}
\end{table}

\section{Conclusions}\label{sec:conclusion}

In this work, based on the weak formulation of PDEs, we propose a deep learning based discontinuous Galerkin method (D2GM) to solve (stochastic) conversation laws. The main idea is that at the
discrete level, the solution is smoother than that at the continuous level. By combining the advantages of discontinuous Galerkin method and deep neural networks, D2GM is able to solve problems with discontinuous solutions over the high-dimensional space. Convergence of the D2GM is proved under some assumptions. This method is tested for PDEs with non-smooth solutions over high-dimensional random space. Over some regime of mesh sizes, D2GM is found to be first-order and second-order accurate in practice. High-order schemes with discontinuous polynomial basis in space can be designed in the same manner. However, how to discretize the temporal derivative with high-order accuracy is unclear at the moment. For example, the leap-frog method is used together with the second-order scheme in space, but the overall second-order accuracy is not observed for the linear conservation law. Therefore, it will be of great interests to desgin high-order schemes for shock waves in the framework of deep neural networks. In summary, the proposed method shows a strong promise for solving high-dimensional uncertain PDEs with discontinuous solutions.

\medskip
\noindent \textbf{Acknowledgment.} This work of J. Chen was supported by National Key R\&D Program of China under grant No. 2018YFA0701700 and No. 2018YFA0701701 and NSFC grant 11971021. The work of S. Jin was supported by Natural Science Foundation of China under grant 12031013.

\bibliographystyle{siam}
\bibliography{D2GM}
\end{document}